\documentclass[11pt]{article}
\usepackage{ArticleStyle}
\usepackage{color,xcolor}
\graphicspath{{figure/}}

\title{Robust Remanufacturing Planning with Parameter Uncertainty }
\author[1]{Zhicheng Zhu}
\author[1]{Yisha Xiang}
\author[2]{Ming Zhao}
\author[1]{Yue Shi}
\affil[1]{Department of Industrial, Manufacturing, and Systems Engineering, Texas Tech University, Lubbock, TX, USA}
\affil[2]{Department of Business Administration, University of Delaware, Newark, DE, USA }
\date{}

\begin{document}
\maketitle
\begin{abstract}
	We consider the problem of remanufacturing planning in the presence of statistical estimation errors. Determining the optimal remanufacturing timing, first and foremost, requires modeling of the state transitions of a system. The estimation of these probabilities, however, often suffers from data inadequacy and is far from accurate, resulting in serious degradation in performance. To mitigate the impacts of the uncertainty in transition probabilities, we develop a novel data-driven modeling framework for remanufacturing planning in which decision makers can remain robust with respect to statistical estimation errors. We model the remanufacturing planning problem as a robust Markov decision process, and construct ambiguity sets that contain the true transition probability distributions with high confidence. We further establish structural properties of optimal robust policies and insights for remanufacturing planning. A computational study on the NASA turbofan engine shows that our data-driven decision framework consistently yields better worst-case performances and higher reliability of the performance guarantee.

\end{abstract}
\section{Introduction}\label{intro}	
The manufacturing industry is a major consumer of materials and energy and imposes a significant impact on environment. Manufacturing activities are responsible for approximately 31\% of the United States’ total energy usage and 19\% of the world’s greenhouse gas emissions within the industrial sector \citep{diaz2010machine, faludi2015comparing}. Sustainable manufacturing with improved environmental performance has drawn a great attention from governments, companies and scientific communities. In the past decade, remanufacturing has emerged as one of the critical elements for developing a sustainable manufacturing industry \citep{ijomah2007development}.  Comparing to manufacturing a new product, remanufacturing   can reduce up to 80\% of energy consumption and carbon dioxide emissions \citep{sutherland2008comparison}, and 40-65\% of manufacturing costs \citep{shi2014product,ford2016additive}. The proceedings of the G7 Alliance on Resource Efficiency noted remanufacturing as one of the most important drivers for getting to a closed-loop economy and improving resource efficiency \citep{EPA}. The growing remanufacturing industry has become an important economic force in the United States with a well-developed market  estimated at \$53 billion per year \citep{lund2012report}. Many products, especially the ones with long lifespans such as turbine blades and automobile parts (e.g., engines,  water pumps), are now routinely remanufactured at the near-end of their life cycles and returned to service \citep{dulman2018maintenance,seitz2007critical}.

Remanufacturing is an industrial process whereby used or broken-down products (or components), referred to as ``cores'', are restored to a like-new condition with an extended lifetime \citep{ostlin2009rm}. During this process, the cores pass through a number of operations including inspection, dismantling, part reprocessing, repair, replacement and reassembly. The performance of the remanufactured cores is expected to meet the desired product standards similar to the original product. Remanufacturing is distinctly different from related activities, such as reuse, repair, and recycling \citep{skrainka2012TurbineRm,lund2012report}. A product is reused when it has not completed its life-cycle and the user decides to stop its use, and a consumer sector is willing to accept it in its current use state, perhaps to its original purpose. Repair fixes what is broken or worn with no attempt to fully restore the product to a like-new condition or to a new life.  Recycling  recovers materials at the end of the product life, returning them into the use stream. Recycling a  complex product (e.g., car), however, can result in a loss of up to 95 percent of the value-added content (e.g., labor, energy) \citep{giuntini2001rm}, in addition to the loss of all functionality of the original product. By reclaiming the material content and retaining the embodied energy and labors  used for manufacturing the original product, remanufacturing is  a more efficient means of resource recirculation than these related activities.

Current remanufacturing practices mainly consider end-of-life products. Such a reactive approach has a number of drawbacks \citep{sutherland2008comparison,song2015proactive}. First, for many products,   remanufacturability is significantly reduced due to no adequate techniques to remanufacture retired products. Second, remanufacturing costs often increase significantly towards the end of a product's life cycle. Third,   substantial negative environmental effects are usually generated because much higher energy and material consumption is required for remanufacturing  due to heavy damage at the end of a product’s life. Several studies have shown that in some cases,  remanufacturing actually consumes more energy than manufacturing a new product \citep{chandler2011worth,gutowski2011worth}. It is paramount that decision makers make informed decisions on the timing of remanufacturing and ensure it is conducted when it is worth the effort.

The robustness of the remanufacturing planning decisions, however, can be threatened by data inadequacy. Since the optimal decision involves suggesting the optimal action, such as no intervention, remanufacturing, or scrapping, at different states (conditions), it is first and foremost to estimate the transition dynamics of a component. The estimation often faces a great challenge because of limited  field data. The situation may be further exacerbated as the field data typically contain a large amount of noises and incorrect information. This data inadequacy poses a critical question to decision makers: How does uncertainty in model parameters translate into uncertainty in the performance of interest? The decision makers must assess whether  any observed nominal improvement in the environmental and economic effects resulted from remanufacturing at certain states is likely to be a true improvement, suggesting remanufacturing in those states, or  conversely, a consequence of the parameter uncertainties due to statistical estimation errors, favoring remanufacturing when it causes negative effects.

Several studies attempted to address the impacts of the uncertainty in model parameters. It has been shown that  when the parameter uncertainty\textemdash the deviation of the model parameters from the true ones, is largely ignored, the optimal policy can lead to serious degradation in performance, because the optimal policy in sequential decision making is quite sensitive to parameter uncertainty, particularly perturbations in the transition probability \citep{mannor2007bias,iyengar2005rmdp,xu2012DRMDP}.  Therefore, with a tacit understanding that the state transition dynamics of a component has to be estimated from historical data in practice, remanufacturing planning is confronted with the \textit{external uncertainty}   due to the deviation of the estimates from their true values in addition to the \textit{internal uncertainty} due to the stochastic nature of a component's condition evolution. To address both internal and external variations and prescribe robust remanufacturing policies, we formulate the remanufacturing planning problem as a robust Markov decision process (RMDP).  We investigate the structural properties of the robust remanufacturing policies in the presence of statistical estimation errors.  We further conduct a computational study based on the operational dataset of the turbofan engine operated by NASA \citep{frederick2007user}. The main contribution of this paper is threefold. 
\begin{itemize}
	\item Provide a novel data-driven modeling framework for remanufacturing planning  in which decision makers can remain robust with respect to statistical estimation errors in transition dynamics.
	\item Establish   structural properties of optimal robust policies and provide insights for decision making in remanufacturing planning in the presence of parameter uncertainties. Our findings also make contributions to the robust Markov decision process literature in which few papers have focused on characterizing the structure of the optimal robust policy.
	\item Conduct computational studies to demonstrate the  optimal robust policies, investigate the out-of-sample performance of the resulting optimal robust remanufacturing policies, and derive data-driven solutions to improve the out-of-sample performance.
\end{itemize}

The remainder of this paper is organized as follows. Section \ref{sec:lit} reviews relevant literature on  remanufacturing planning and sequential decision making with parameter uncertainty. In Section \ref{sec:model}, we formulate the remanufacturing planning problem as a RMDP. Section \ref{sec:policy} establishes conditions to ensure  the optimal robust policies are of control-limit type. In Section \ref{sec:computationalStudy}, we present a computational study using simulated operational data of  NASA's turbofan engines.  Section \ref{sec:conclusion} concludes this study and suggests future research  directions.

\section{Literature Review}\label{sec:lit}
Our study is related to two streams of the literature: remanufacturing planning and sequential decision-making with parameter uncertainty.

\subsection{Remanufacturing Planning}

The management of remanufcturing  production and control activities greatly  differs  from management activities in traditional manufacturing. Production planning and control activities are more complex for remanufacturing firms due to complicating characteristics such as the uncertain timing and quantity of returns \citep{guide2000production}. The majority of literature in  remanufacturing operational management focus on production planning and control activities, such as reverse logistic management, balancing returns and demands, and inventory control. \cite{Dekker1999inventory} consider production planning and inventory control in systems where manufacturing and remanufacturing operations occur simultaneously. \cite{savaskan2004closed} identify and model three close-loop supply chains to  address the problem of choosing the appropriate reverse channel structure for the collection of used products. \cite{galbreth2010optimal} study the optimal core acquisition quantity problem subject to the uncertainty of core conditions. We refer the readers to a review \citep{govindan2015review} and the papers therein for more research works relevant to  reverse logistics and closed-loop suply chain.

Remanufacturing planning, while being recognized, has received little attentions. Existing works on remanufacturing timing decisions often either ignore both types of uncertainties in transition dynamics of a remanufacturing system or only focus on the internal variation. For example, \cite{song2015proactive} determine remanufacturing timing based on a \textit{deterministic} degradation process charaterized by residual strength factors. \cite{wang2016decision} recommend  remanufacturing  based on online monitoring: Products are remanufactured when it reaches the limit condition   beyond which the product is no longer remanufacturable. External variation is largely ignored, and hence, remanufacturing could be blindly suggested even if it might lead to increased negative environmental or economic impacts, resulting in the robustness of remanufacturing planning decisions in question.

Remanufacturing planning bears a close resemblance to maintenance planning which aims to determine the optimal timing of preventive maintenance.  In this paper, we use Markov models for remanufacturing planning; the most relevant works in maintenance optimization literature are the ones that model maintenance problems using a Markov decision process (e.g., \cite{Kharoufeh2010EJOR}, \cite{elwany2011structured},  \cite{Makis2013OR}). Most maintenance optimization models that are formulated as a Markov decision process (MDP), however,  assume that the cost parameters and the transition kernel are known, and hence, cannot provide satisfactory out-of-sample performances when future realizations deviate from the predicted ones. One of the few papers that consider ambiguity in transition probabilities is by \cite{kim2016robustCBM}. In his paper, \cite{kim2016robustCBM} considers a failing system whose underlying state is unobservable and accounts uncertainties in both posterior distributions and transition probabilities. The optimization of an ambiguous partially observable MDP (APOMDP) is generally challenging even for small state spaces. For remanufacturing planning, the state space would become large, rendering the optimization of the resulting APOMDP computationally prohibitive.

\subsection{Sequential Decision-Making with Parameter Uncertainty}

Conventional MDP assumes that the transition probabilities  and rewards  can be estimated confidently and focuses solely on the uncertainty stemming from the stochastic nature of MDPs.  In practice, however, estimating the true transition probabilities is difficult, if not impossible, because of the limited data availability or inevitable statistical errors or both.  When  parameters deviate from the true ones and such uncertainty is  ignored, the optimal policy can lead to serious inferior performance, because the optimal policy in sequential decision making is quite sensitive to parameter uncertainty, particularly deviations in the transition probability \citep{mannor2007bias,iyengar2005rmdp,xu2012DRMDP}.

Early works on the MDPs with parameter uncertainties, including \cite{silver1963markovian, satia1973markovian,white1986impreciseReward} and \cite{white1994markov}, formulate the uncertainty in either a game-theoretic or Bayesian approach.  In the game-theoretic formulation, it is assumed that the uncertainty about the transition probabilities is encoded by describing the set of all transition probability rows. Hence, when the decision maker makes a decision for a given state, the nature selects a transition probability row from the set to minimize the reward.  \cite{satia1973markovian} use the game-theoretic formulation to model the transition uncertainty in MDP and proposed a policy iteration procedure to solve the problem. \cite{white1994markov} further develop a modified policy iteration-based algorithm for the MDP with imprecise transition probabilities. The Bayesian approach, first introduced by \cite{silver1963markovian}, assumes a known priori probability distribution of each transition probability row. Hence, the transition probabilities can be updated along the Bellman's equations.  Dirichlet priors are a common choice of modeling the uncertainty in transition probabilities \citep{delage2010percentile}.

Most of the early contributions, however, do not concern the  construction of ambiguity sets. Inspired by the data-driven approaches, recent RMDP works \citep{iyengar2005rmdp,nilim2005RMDP,wiesemann2013RMDP} have developed various methods to construct the uncertainty set of transition probabilities that contain the true transition probabilities with high confidence. Many statistical methods, such as likelihood constraints, deviation-type constraints and distance metrics (e.g., Wasserstein ball, \(\phi-\)divergence balls), have been applied to construct an uncertainty set of transition probabilities with historical samples \citep{iyengar2005rmdp,nilim2005RMDP,wiesemann2013RMDP}. Reformulation of RMDPs with different types of ambiguity sets and the corresponding tractability have also been studied in the literature.  Compared to the theoretical orientation of these works, our present work focuses more narrowly on developing methods for a specific problem class, establishing structural properties of optimal robust policies, and  providing executable insights.

\section{ Robust Remanufacturing Planning Problem}\label{sec:model}
Consider remanufacturing planning of a single-unit system that degrades during its operation. Because we focus on single-unit systems, the words system and component are used interchangeably throughout the paper. The component is inspected at equally spaced discrete time epochs \(\mathcal{T}=\{0,1,\dots\}\). Let \((\mathcal{S}, \mathcal{K})\)  be the state space, where  \(\mathcal{S}=\{0,1,2,...,S\}\) represents the set of conditions and \(\mathcal{K}=\{0,1,\dots\}\)  represents the set of cumulative numbers of completed remanufacturing activities. Note that a larger state denotes a worse condition and the worst state \(S\) is an absorting state, meaning the systems stays there if there is no intervention. At  each epoch, a decision maker oberves the state of the component and then chooses an action from the  set \(\mathcal{A}=\{0,1,2\}\), where 0 means continuing operation to the next observation time, 1 means remanufacturing, which takes one period, and 2 means scraping the component.

The most important objective of remanufacturing is to minimize the negative environmental impacts while sustaining profitable growth. The direct environmental impacts of a manufactured/remanufactured system are often measured by greenhouse gas  emissions (e.g., CO\textsubscript{2}, CH\textsubscript{4}, N\textsubscript{2}O, etc.) using  life cycle assessment (LCA). LCA is a technique that compiles an inventory of relevant energy consumption and material inputs and environmental releases, and then evaluates the potential environmental impacts associated with identified inputs and releases \citep{curran2011scientific}.  Such an evaluation can be done using published databases (e.g., EcoScan)  or commercial softwares (e.g., GaBi) if the system is complicated. There are several types of LCA. For instance, cradle-to-grave is the assessment of a full product life cycle from resource extraction (``cradle") to use phase and disposal phase (``grave"), cradle-to-gate is an assessment of a partial product life cycle from resource extraction to the factory gate (i.e., before it transported to the consumer), and gate-to-gate, also a partial LCA, evaluates the eco-burden of a manufacturing facility. In this paper, rather than model different types of greenhouse gas emissions and formulate a multi-objective problem, we model the environmental impacts using carbon cost, which is determined by the amount of carbon emissions and the carbon price. As carbon trading increasingly recognized as one of the most effective approaches to incentivising companies to become environmental friendly \citep{abdallah2012green} and more carbon trading systems  established around the world, remanufacturing planning models that consider the carbon costs will become more relevant and applicable.

To model the profit of a remanufacturing system, we assume that during each decision period, the decision maker receives a gain \(g(s,k)\) (e.g., production revenue) if operation is not interrupted and incurs some environmental costs   \(e(s,k)\). The reward of keeping operation in one period is thus denoted by \(r(s,k)=g(s,k)-e(s,k)\).   If the decision is to remanufacture the component, a remanufacturing cost \(c_r\), which comprises the manufacturing and carbon costs, is incurred. If the system is scrapped, a salvage value \(c_s\) is received. We assume that $c_r$ and  $c_s$ are the same regardless of a component's condition.

Although remanufactuirng is supposed to restore a component to like-new conditions, each remanufacturing process typically makes the component less resistant to deterioration. Hence, a component's deterioration process depends on the cumulative number of completed remanufacturing activities, which is modeled as follows. For a system that has been remanufactured \(k\) times, let \(\pmb{P}=[p(s^\prime|s,k)]_{s, s^\prime \in \mathcal{S}, k\in \mathcal{K} }\)  denote the transition probability matrix when the decision is to wait.   Note that remanufacturing brings the system status to an as-good-as-new  condition, but increments the cumulative number of remanufacturings by one.  We assume that the deterioration process is irreversible, and hence \(\pmb{P}\) is an upper triangular matrix.  Due to limited data availability and statistical estimation errors, the transition probability of a remanufacturing system is fundamentally unknown.  To mitigate the effects of uncertain transition probabilities, we assume that the true transition kernel  is contained in an ambiguity set. 

Next, we present an important assumption regarding the ambiguity set, which ensures deterministic and Markovian policies \citep{iyengar2005rmdp}. 

\begin{assumption} [Rectangularity]
	An RMDP problem has a rectangular  ambiguity set if the ambiguity set has the form  \(\mathcal{U}=\bigotimes_{s \in\mathcal{S},k\in\mathcal{K}}\mathcal{U}_{sk}\) where \(\bigotimes\)	stands for the Cartesian product, and \(\mathcal{U}_{sk}\) is the projection of \(\mathcal{U}\) onto the parameters of state \((s,k)\).	
\end{assumption}

The implications of the rectangularity assumption is often interpreted in an adversarial setting \citep{iyengar2005rmdp,nilim2005RMDP}: The decision maker first chooses a policy \(\pi\). Then an adversary observes \(\pi\), and chooses a distribution that minimizes the reward. In this context, rectangularity is a form of an independence assumption: The choice of a particular distribution  for a given state $(s,k)$ does not limit the choices of the adversary of other states. There are two possible models for transition matrix uncertainty. In the first model, referred to as the stationary uncertianty model, the transition matrices chosen by the adversary depending on the policy  once and for all, and remain fixed thereafter. In the second model, referred to as the time-varying uncertainty model, the transition matrices can vary arbitrarily with time,  within their prescribed bounds. It has been shown in \cite{nilim2005RMDP} that for  a finite horizon problem with a discounted cost function, the gap between the optimal value of the stationary uncertainty problem and that of its time-varying counterpart goes to zero as the horizon length goes to infinity. In this paper, we consider the stationary worst-case distribution, that is, the choices of $\pmb p(\cdot|s,k)$ are the same every time the state $(s,k)$ is encountered.   Note that there is no ambiguity in transitions in the period during which remanufacturing is conducted, since remanufacturing takes one period and there is no transition in that period.  Because the optimal robust policies of the remanufacturing planning are Markovian and deterministic under the rectangularity assumption, we have  the robust remanufacturing planning optimization model in the following recursive form:
\begin{equation}
\label{Bellman}
V(s,k) =\sup_{a \in \mathcal{A}}w(s,k;a), \tag{RRmPO}
\end{equation}
where
\begin{equation*}
w(s,k;a)=
\left\{
\begin{array}{ll}
\inf_{\pmb{P} \in \mathcal{U}} r(s,k)+\beta \sum_{s^\prime \in \mathcal{S}} p(s^\prime|s,k)  V(s^\prime,k), & a=0,\\
-c_\text{r}+\beta V(0,k+1), & a=1, \\
c_\text{s}, & a= 2.
\end{array}
\right. 
\end{equation*}

\subsection{Construction of Ambiguity Sets} \label{sec:construction}

As mentioned in the introduction, the motivation for the robust methodology is the presence of the statistical errors associated with estimating the transition probabilities using historical data. A natural choice for the ambiguity set is the confidence regions associated with density estimation. We thus use the Kullback-Leibler (KL) distance to construct the ambiguity set around the empirical transition probabilities. Bootstrap resampling, a common non-parametric method to address the uncertainty, is also used to construct confidence intervals of estimators. 

An important property of the ambiguity set constructed using  these two methods is that it cannot pop a scenario while allows scenario suppression. A scenario that never occurs in the nominal problem cannot have a positive probability (or, ``pop") in the ambiguous problems. Such a property ensured that the assumption of  no-self-improving in system's deterioration   is held in the ambiguous problem.  More importantly, the ambiguity sets constructed using these two methods are convex, and hence leads to a  computational tractable problem. In addition, the KL distance and bootstrapping are already being used in statistics, making them attractive  to deal with data directly.

\textbf{KL Distance.}  In a data-driven setting, the empirical distribution typically serves as the nominal distribution.   For notational convenience, we drop the notations of \(k\) and \(a\). Let $\hat{\pmb{p}}_{s}$ be the maximum likelihood estimator of $\pmb{p}_{s}$  given a state \(s\in\mathcal{S}\):
\begin{equation}
	\hat{p}_{s}(s')=\frac{n(s^\prime|s)}{\sum_{s^\prime \in \mathcal{S}}n(s^\prime|s)},
\end{equation}
where \(n(s^\prime|s)\) is the number of transitions from state \(s\) to \(s^\prime\). The KL distance between $\hat{\pmb{p}}_{s}$ and $\pmb{p}_{s}$ is defined as 
\begin{equation}
	\label{KL}
	D(\pmb{p}_{s}||\hat{\pmb{p}}_{s}) = \sum_{s^\prime\in\mathcal{S}}p_{s}(s^\prime)\log\left(\frac{p_{s}(s^\prime)}{\hat{p}_{s}(s^\prime)}\right).
\end{equation}
It is obvious that $D(\pmb{p}_{s}||\hat{\pmb{p}}_{s})\geq 0$ with equality holds when $\pmb{p}_{s}=\hat{\pmb{p}}_{s}$. Given a state \(s\in\mathcal{S}\), the KL-distance-based ambiguity set is given by
\begin{equation}	
	\mathcal{U}_{s} =\left\{\pmb{p}_s:	D(\pmb{p}_{s}||\hat{\pmb{p}}_{s})\leq \theta,\sum_{s^\prime \in \mathcal{S}}p_s(s^\prime)=1,p_s(s^\prime)\in[0,1],s^\prime\in\mathcal{S}\right\}.
\end{equation}
Let \(N_s=\sum_{s^\prime \in \mathcal{S}}n(s^\prime|s)\). It has been shown that the normalized estimated KL-distance \(2N_s D(\pmb p_s||\hat{\pmb{p}}_s)\) asymptotically  follows a \(\chi_{|\mathcal{S}|-1}^2\) distribution (see more details in \cite{ben2013robust}). We thus have the following (approximate) (1-\(\alpha\))-confidence set around \(\hat{\pmb{p}}_s\)
\begin{equation}
	\mathcal{U}_s=\{\pmb{p}_s: 	D(\pmb{p}_{s}||\hat{\pmb{p}}_{s})\leq \theta\} 
\end{equation}
where \(\theta=\chi_{|\mathcal{S}|-1, 1-\alpha}^2/(2N_s) \). 

\textbf{Bootstrap Resampling}. Let \(\boldsymbol{D}\) be  realizations of the Markov chain \(\{S_n;n\geq 0\}\) with transition matrix \(\pmb P\) and \(\hat{\pmb P}\) be the maximum likelihood estimator of \(\pmb P\) based on the observed data \(\boldsymbol{D}\). For a large number of statistics of interest, the bootstrap distribution approximates  the sampling distribution (i.e., asymptotic normality of \(\sqrt{N}(\tilde{\pmb P}-\hat{\pmb P})\) where \(\tilde{\pmb p}\) is the bootstrap estimator of \(\pmb P\)).      From the bootstrap distribution, one can assess the uncertainty of each probability in the transition matrix and construct the confidence interval for each probability.  The bootstrap distribution can be calculated by direct theoretical calculation, which draws samples with replacement, or Monte Carlo approximation. The reader is referred to \cite{efron1994introduction} for more detailed bootstrapping procedures. Note that the ambiguity set constructed by bootstrap resampling has the same form of the ambiguity set constructed using the interval matrix  method in \cite{nilim2005RMDP}. Therefore, we also refer to the method of constructing ambiguity sets using bootstrap sampling as the interval matrix method hereinafter since it better describes the form of the ambiguity sets.

\section{Structure of the Optimal Robust Policy} \label{sec:policy}
In this section, we investigate the structural properties of the optimal robust remanufacturing policies. We will focus our attention on control-limit policies. We establish sufficient conditions that ensure the existence of monotonically control-limit policies. The optimality of such structured policies is important because they are appealing to decision makers and enables efficient computation and are easy to implement.  Our analysis will make significant use of the notion of the stochastic dominance, which helps establish stochastic dominance relationships for transition behaviors.  Below, we define some stochastic order concepts that are used in our analysis.

\begin{definition}
	\mbox{} 
	\label{definition01}
	\begin{itemize}
		\item[\textup{(a)}] A transition probability matrix \(\pmb{P}=[p(i|j)]_{i,j=0,1,...n}\) is said to be IFR (increasing failure rate)  if \(\sum_{i=m}^{n}p(i|j)\) is non-decreasing in \(j\) for all \(m=0,1,...,n\).
		
		\item[\textup{(b)}] For two transition probability matrices \( \pmb{P}_1=[p_1(i|j))]_{i,j=0,1,...n}\) and \(\pmb{P}_2=[( p_2(i|j))]_{i,j=0,1,...n}\), we say \(\pmb{P}_1\) dominates \(\pmb{P}_2\), \(\pmb{P}_1\succeq \pmb{P}_2\), if \(\sum_{i=m}^{n}p_1(i|j) \geq \sum_{i=m}^{n}p_2(i|j)\) for all \(j,m=0,1,...,n\).
	\end{itemize}

\end{definition}

\begin{assumption}
	\label{asm:transition}
	\mbox{}
	Let \(\hat{\pmb{P}}(\cdot|\cdot,k)\) denote the nominal transition probability matrix for a system that has been remanufactured \(k\) times, 	
	\begin{itemize}
		\item [\textup{(a)}]  \(\hat{\pmb{P}}(\cdot|\cdot,k)\) is IFR for all \(k\in \mathcal{K}\).
		\item [\textup{(b)}] \(\hat{\pmb{P}}(\cdot|\cdot,k+1) \succeq\hat{\pmb{P}}(\cdot|\cdot,k) \) for all \(k\in \mathcal{K}\). 
	\end{itemize} 
\end{assumption}
Assumption \ref{asm:transition}(a) implies that, given the  cumulative number of completed remanufacturing activities \(k\), the system in a worse state at the current epoch is more likely than the other to be found in a worse condition at the next epoch. Assumption \ref{asm:transition}(b) imposes a first-order stochastic dominance relationship  among the system's deterioration matrices corresponding to different remanufacturing histories. More explicitly, given two systems with   the same condition but different remanufacturing histories, the system with a larger \(k\) is more likely to get worse than the other during operation. Additional assumption is made  regarding the operational gains, environmental costs, and the salvage value. 

\begin{assumption}
	\mbox{} 
	\label{asm:coststructure}
	\begin{itemize}
		\item[\textup{(a)}] The operational gain \(g(s,k)\) is non-increasing in \(s\in \mathcal{S}\) and \(k\in\mathcal{K}\), and the carbon cost \(e(s,k)\) is non-decreasing in \(s\in \mathcal{S}\) and \(k\in\mathcal{K}\);
		\item[\textup{(b)}] The reward at state \(S\), the salvage value \(c_s\) and the discount factor \(\beta\) satisfy the following condition: \( \dfrac{r(S,0)}{1-\beta}<c_\text{s}\).
	\end{itemize}
\end{assumption}
Assumption \ref{asm:coststructure}(a) implies that as the number of completed remanufacturing activities increases and its condition worsens, the gain decreases and the carbon cost increases. For example, an engine in a worse state usually incurs higher maintenance costs and consumes more gasoline or electricity.  Assumption \ref{asm:coststructure}(b) ensures that the decision of no intervention (i.e., \(a=0\)) is excluded when a  system is at the worst state  for all \(k \in \mathcal{K}\) because it is not practical that the system stays in the worst condition \(S\)  for an infinitely long time. This unrealistic scenario is eliminated by assuming that the  total   expected reward from doing nothing at state \((S,0)\),  computed as  \(\sum_{t=0}^{\infty}\beta^t r(S,0)=\dfrac{r(S,0)}{1-\beta}\),  is  less than  the salvage value. Since \(r(S,0) \ge  r(S,k)\) for all \(k>0\), the condition also eliminates the no-intervention option for state \((S,k)\) for all \(k>0\).

\subsection{Remanufacturing Planning with KL-Distance-Based Ambiguity Sets}

In this section, we  consider Model \eqref{Bellman} with KL-distance-based ambiguity sets. We  provide reformulations and  establish conditions that ensure control-limit type policies. We first reduce the  bi-level problem (\ref{Bellman})  to a single-level problem by applying the Langurangian dual theory, and then investigate the structure of the robust value function, which is necessary for establishing control-limit robust remanufacturing planning policies.

\begin{proposition}
	\label{prop:reformulation-KL}
	For Model \eqref{Bellman} with  KL-distance-based ambiguity sets, \(w(s,k;0)\) can be reformulated as 
	\begin{equation}
	\label{reformulation-KL-a0}
	w(s,k;0) = \sup_{\mu > 0} r(s,k)- \mu\log\left(\sum_{s'\in\mathcal{S}}\hat{p}(s'|s,k) \exp\left(\dfrac{-\beta V(s',k)}{\mu}\right)\right) -\mu\theta,
	\end{equation}
	and the worst-case distribution is
	\begin{equation}
	\label{theorem04_00b}
	p^*(s'|s,k)= \dfrac{\hat{p}(s'|s,k)\exp\left(\dfrac{-\beta V(s',k)}{\mu_{sk}^*}\right)}{\sum_{s'\in\mathcal{S}}\hat{p}(s'|s,k) \exp\left(\dfrac{-\beta V(s',k)}{\mu_{sk}^*}\right)},
	\end{equation}
	where \(\mu_{sk}^*\) is the optimal solution of the dual problem \eqref{reformulation-KL-a0} given \(s\) and \(k\).
\end{proposition}
\begin{proof}
	See Appendix A.1.
\end{proof}

\begin{proposition}
	\label{prop:Monotone-ValueFunction-KL}
	For Model \eqref{Bellman} with ambiguity sets constructed using the KL distance,  the value function \(V(s,k)\) is non-increasing in  \(s \in \mathcal{S}\) and \(k\in\mathcal{K}\).
\end{proposition}
\begin{proof}
	See Appendix A.2.
\end{proof}

Next, we establish conditions that ensure control-limit robust policy structures, that is, the remanufacturing decisions  are of control-limit type with respect to the condition of the system and the cumulative number of completed remanufacturing activities. 

\begin{theorem} 
	\label{thm:Monotone-S-KL} 
	For Model \eqref{Bellman} with KL-distance-based ambiguity sets,   there exists a  cumulative number of completed remanufacturing activities \(k^\ast \in \mathcal{K}\), and an  operational state \(s_k \in \mathcal{S}\) such that for \(k <k^\ast\)
	\begin{equation*}
	a(s,k)=\Bigg\{\begin{array}{ll}
	0 & \text{if } s < s_k,\\
	1 & \text{if } s\geq s_k,
	\end{array}
	\end{equation*}
	and for \(k\geq k^\ast\)
	\begin{equation*}
	a(s,k)=\Bigg\{\begin{array}{ll}
	0 & \text{if } s < s_k,\\
	2 & \text{if } s\geq s_k.
	\end{array}
	\end{equation*}
\end{theorem}
\begin{proof}
	See Appendix A.3.
\end{proof}

Theorem \ref{thm:Monotone-S-KL} shows that for a system that has \(k<k^\ast\), the optimal decision is either wait until the next period or remanufacture, and the system is remanufactured when the condition is equal to or exceeds a limit. When the cumulative number of completed remanufacturing activities reaches the threshold \(k^\ast\), the optimal decision is either wait until the next period or scrap and there exists a scrapping threshold. This implies that despite the cost savings and environmental benefits that make remanufacturing appealing, remanufacturing activity is not recommended after being conducted  certain number of times. Note that \(k^\ast=0\) is a special case that remanufacturing is not optimal for all \(k \in \mathcal{K}\). Let \(\zeta_{\text{rm}}(k)\) denote the control limit \(s_k\) for \(k<K^\ast\) and \(\zeta_{\text{scrap}}(k)\) denote the the control limit \(s_k\) for \(k\ge K^\ast\).  We further examine the structure of \(\zeta_{\text{rm}}(k)\) and \(\zeta_{\text{scrap}}(k)\) in the next theorem.

\begin{theorem} 
		\label{thm:Monotone-K-KL} 
		Consider  Model \eqref{Bellman} with the KL-distance-based ambiguity sets. Then, the following holds:
		\begin{itemize}
			\item [\textup{(a)}] 	
			
			If
			\(
			\dfrac{\beta r(0,0)}{1-\beta}-\beta c_s \leq r(s,k) - r(s,k+1)
			\), 
			\(\zeta_{\textup{rm}}(k)\) is non-increasing in \(k,k<k^*\).
			\item [\textup{(b)}] \(\zeta_{\textup{scrap}}(k)\) is non-increasing in \(k, k\ge k^*\).
		\end{itemize}
	
\end{theorem}

\begin{proof}
	See Appendix A.4.
\end{proof}

The condition in Theorem \ref{thm:Monotone-K-KL}(a) is  restrictive. We will show that most violations do not change the monotone structure of \(\zeta_{\text{rm}}(k)\) and \(\zeta_{\text{scrap}}(k)\) in Section \ref{sec:violationTest} through some computational studies. 

\subsubsection{Solution Methodology for KL Distance Model}\label{sec:robustValueIter}

Model \eqref{Bellman} can be efficiently solved using robust value iteration as described below:

\begin{algorithm}[h]
	\caption{Robust Value Iteration}
	{\small{
			\begin{algorithmic}[1]
				\State \textbf{Initialization:}\par  
				\(\bar{V}(s,k), a^\ast(s,k) \leftarrow0,  V(s,k) \leftarrow M,  \forall (s,k)\in\mathcal{S}\times\mathcal{K}\),
				\(\epsilon > 0\)
				\While{\(||\bar{\pmb V}- \pmb V||\geq \frac{(1-\beta)\epsilon}{4\beta}\)}
				\State 	\(\pmb V \leftarrow \bar{\pmb V}\)
				\For {\((s,k)\in\mathcal{S}\times\mathcal{K}\)}
				\State \(\bar{V}(s,k)\leftarrow \max_{a\in A}w(s,k;a)\)
				\State \(a^\ast(s,k)\leftarrow \arg\max_{a\in A}w(s,k;a) \)
				\EndFor
				\EndWhile
				\State \Return \(\bar{\pmb V}, \pmb a^\ast\)
	\end{algorithmic}}}
	\label{alg01}
\end{algorithm}
To solve the inner problem \(w(s,k;0)\) in step 5, we can either employ a numerical search for the dual problem in Equation \eqref{reformulation-KL-a0}, or solve the primal problem as the following conic program:

\renewcommand{\baselinestretch}{0.0}
\begin{align}
w(s,k;0) & = r(s,k)+\min_{\pmb p(\cdot|s,k)\in\mathcal{U}_{sk}} \sum_{s^\prime\in\mathcal{S}}p(s^\prime|s,k)V(j,k) \nonumber\\
s.t. \quad & p(s^\prime|s,k) \geq 0, s^\prime\in\mathcal{S}, \sum_{s^\prime\in\mathcal{S}}p(s^\prime|s,k)  = 1 \nonumber\\
\quad &\sum_{s^\prime\in\mathcal{S}}p(s^\prime|s,k)\log \left(\dfrac{p(s^\prime|s,k)}{\hat{p}(s^\prime|s,k)}\right) \leq \theta \label{KL-conic-constraint01},
\end{align}
\renewcommand{\baselinestretch}{1.5}\\
where constraint (\ref{KL-conic-constraint01}) can be represented by the following exponential cone:
\begin{align*}
&\sum_{s^\prime\in\mathcal{S}}z(s^\prime) =  \theta,\\
&\big(\hat{p}(s^\prime|s,k(s^\prime), p(s^\prime|s,k), -z(s^\prime)\big) \in K_{\text{exp}}, s^\prime\in\mathcal{S},
\end{align*}
where \(K_{\text{exp}}=\{(x_1,x_2,x_3): x_1 \ge x_2e^{x_3/x_2}, x_2>0\}\) is the exponential cone in \(\mathbb{R}^3\).

\subsection{Remanufacturing Planning with Interval-Matrix-Based Ambiguity Sets}
In this section, we establish conditions that guarantee structural properties of the optimal robust policies of Model \eqref{Bellman} with ambiguity sets  constructed using the interval matrix model.  The interval matrix model describes the uncertainty on the rows of the transition matrices  in the form \(\mathcal{U}_{sk}=\{\pmb{p}(\cdot|s,k):\underline{p}(s'|s,k)\leq p(s'|s,k)\leq \bar{p}(s'|s,k), \sum_{s' \in \mathcal{S}}p(s'|s,k)=1\}\).  Note that since components in each row vector of a transition matrix are constrained by \(\pmb{p}(\cdot|s,k)^T\textbf{1}=1\), more effective lower and upper bounds of the transition probability of each state $(s,k)\in \mathcal{S}\times \mathcal{K}$ can be obtained. The effective upper bounds are
\begin{equation}
\label{upperbound}
\min\big\{\bar{p}(s'|s,k),1-\sum_{s''\in \mathcal{S}\backslash \{s'\}}\underline{p}(s''|s)\big\}, \forall s'\in \mathcal{S}.
\end{equation}
and the effective lower bounds are
\begin{equation}
\label{lowerbound}
\max\big\{\underline{p}(s'|s,k),1-\sum_{s''\in \mathcal{S}\backslash \{s'\}}\bar{p}(s''|s,k)\big\}, \forall s'\in \mathcal{S},
\end{equation}
The upper and lower bounds hereinafter are referred to the bounds defined in Equations (\ref{upperbound}) and (\ref{lowerbound}). We first establish conditions that ensure the value function is monotone and the worst-case transition matrices are IFR, and then examine the structures of the robust optimal remanufacturing planning policies.
\begin{proposition}
	\label{prop:Monotone-ValueFunction-S-IM}
	Consider Model \eqref{Bellman} with the ambiguity set constructed using the interval matrix model. 
	
	\begin{itemize}
		\item [\textup{(a)}] If the lower and upper bounds of all states \((s,k)\in \mathcal{S}\times\mathcal{K}\) satisfy  the following conditions,
		\begin{align}
		& \sum_{s''=0}^{i} \underline{p}(s''|s,k) \geq  \sum_{s''=0}^{i} \underline{p}(s''|s',k), \forall i \in \mathcal{S},  k\in \mathcal{K}, \text{if } s'\ge s \label{lowerboundConstraints-S}, \\
		& \sum_{s''=i}^{S} \bar{p}(s''|s,k) \leq  \sum_{s''=i}^{S} \bar{p}(s''|s',k), \forall i \in \mathcal{S}, k\in \mathcal{K}, \text{if } s'\ge s \label{upperboundConstraints-S}
		\end{align}
		
		then the value function \(V(s,k)\) is non-increasing in  \(s\) for all \(k \in \mathcal{K}\).
		\item[\textup{(b)}] If the lower and upper bounds of all states \((s,k)\in \mathcal{S}\times\mathcal{K}\) satisfy conditions (\ref{lowerboundConstraints-S}) and (\ref{upperboundConstraints-S}), and  the following conditions,
		\begin{align}
		& \sum_{s''=0}^{i} \underline{p}(s''|s,k) \geq  \sum_{s''=0}^{i} \underline{p}(s''|s,k'), \forall i \in \mathcal{S}, s\in\mathcal{S}, \text{if } k'\ge k  \label{lowerboundConstraints-K}\\
		& \sum_{s''=i}^{S} \bar{p}(s''|s,k) \leq  \sum_{s''=i}^{S} \bar{p}(s''|s,k'), \forall i \in \mathcal{S}, s\in\mathcal{S}, \text{if } k'\ge k  \label{upperboundConstraints-K}
		\end{align}
		then the value function \(V(s,k)\) is non-increasing for all \(k \in \mathcal{K}\).	
	\end{itemize}
	
\end{proposition}
\begin{proof}
	See Appendix A.5.
\end{proof}

Proposition \ref{prop:Monotone-ValueFunction-S-IM} provides a set of conditions sufficient to guarantee the structure of the value function and the worst-case transition matrices, which are critical to the existence of control-limit robust optimal policies. These conditions are generally not restrictive, and are expected to be roughly satisfied. Conditions (\ref{lowerboundConstraints-S}) and (\ref{upperboundConstraints-S}) ensure that the worst case transition matrices are IFR, and conditions (\ref{lowerboundConstraints-K}) and (\ref{upperboundConstraints-K}) guarantee that the worst transition matrix $\pmb{P}^\ast(\cdot|\cdot,k+1)$ 
dominates the worst transition matrix $\pmb P^\ast(\cdot|\cdot,k)$ for all \(k\in \mathcal{K}\). 

In the proof of Proposition \ref{prop:Monotone-ValueFunction-S-IM}, we also obtain the transition behaviors of worst-case transition probability matrices when the decision is to wait, which are summarized in the following corollaries.
\begin{corollary}
	\label{cor:worstcaseditr-IM}
	For Model  \eqref{Bellman} with the ambiguity set constructed using the interval matrix model, if the lower and upper bounds of all states \((s,k)\in \mathcal{S}\times\mathcal{K}\) satisfy  conditions (\ref{lowerboundConstraints-S}) and (\ref{upperboundConstraints-S}), then  the worst-case transition matrices are:
	\begin{equation}
	p^*(s'|s,k)=
	\left\{
	\begin{array}{ll}
	\bar{p}(s'|s,k), & s' > \delta_{sk},\\
	\underline{p}(s'|s,k), & s' < \delta_{sk}, \\
	1-\sum_{s'=1}^{\delta_{sk}-1}\underline{p}(s'|s,k) -\sum_{s'=\delta_{sk}+1}^{S}\bar{p}(s'|s,k)  , & s' = \delta_{sk},\\
	\end{array}
	\right.
	\end{equation} 
	where \(\delta_{sk}=\min\{\delta\in\mathcal{S}:\sum_{s'=0}^{\delta}\underline{p}(s'|s,k)+\sum_{s'=\delta+1}^{S}\bar{p}(s'|s,k) \leq 1\}\) for all \(s\) and \(k\).	
\end{corollary}
\begin{proof}
	See Appendix A.6.
\end{proof}

Let $\pmb{P}^\ast(\cdot|\cdot,k)$ denote the worst-case transition matrices in Corollary \ref{cor:worstcaseditr-IM}. Corollary \ref{cor:worstcaseditr-structure-IM} summarizes the structure of $\pmb{P}^\ast(\cdot|\cdot,k)$.

\begin{corollary}
	\label{cor:worstcaseditr-structure-IM}
	$\pmb{P}^\ast(\cdot|\cdot,k)$ is IFR for all \(k\in \mathcal{K}\); if \(k_1>k_2\) (\(k_1, k_2 \in \mathcal{K}\)), \(\pmb{P}^\ast(\cdot|\cdot,k_1) \succeq \pmb{P}^\ast(\cdot|\cdot,k_2)\).
\end{corollary}

Next we show that similar to optimal robust remanufacturing policies for model with KL-distance-based ambiguity sets, the optimal robust policies for model with interval-matrix-based ambiguity sets  exhibit the same control-limit structure with respect to \(s\) and \(k\) under some conditions.

\begin{theorem} 
	\label{thm:Monotone-S-IM} 
	For Model \eqref{Bellman} with the ambiguity set constructed by the interval matrix model, if the lower and upper bounds of all states \((s,k)\in \mathcal{S}\times\mathcal{K}\) satisfy  conditions (\ref{lowerboundConstraints-S}) to (\ref{upperboundConstraints-K}), then there exist a  cumulative number of remanufacturings \(k^\ast \in \mathcal{K}\) and an  operational state \(s_k \in \mathcal{S}\) such that for \(k <k^\ast\)
	\begin{equation*}
	a(s,k)=\Bigg\{\begin{array}{ll}
	0 & \text{if } s < s_k,\\
	1 & \text{if } s\geq s_k,
	\end{array}
	\end{equation*}
	and for \(k\geq k^\ast\)
	\begin{equation*}
	a(s,k)=\Bigg\{\begin{array}{ll}
	0 & \text{if } s < s_k,\\
	2 & \text{if } s\geq s_k.
	\end{array}
	\end{equation*}
\end{theorem}
\begin{proof}
	See Appendix A.7.
\end{proof}

\begin{theorem} 
		\label{thm:Monotone-K-IM} 
	 For Model \eqref{Bellman} with the ambiguity set constructed using the interval matrix model, if the lower and upper bounds of all states \((s,k)\in \mathcal{S}\times\mathcal{K}\) satisfy  conditions \eqref{lowerboundConstraints-S} to \eqref{upperboundConstraints-K}, then the following holds:
		\begin{itemize}
			\item [\textup{(a)}] 
					
			If
			\(
			\dfrac{\beta r(0,0)}{1-\beta}-\beta c_\text{s} \leq r(s,k) - r(s,k+1),
			\)
			is satisfied, \(\zeta_{\textup{rm}}(k)\) is non-increasing in \(k,k<k^*\).
			\item [\textup{(b)}] \(\zeta_{\textup{scrap}}(k)\) is non-increasing in \(k, k\ge k^*\).
		\end{itemize}
	\end{theorem}

\begin{proof}
	See Appendix A.8.
\end{proof}

\subsubsection{Solution Methodology for Interval Matrix Model}
Model \eqref{Bellman} with the ambiguity set constructed using the interval matrix model can be efficiently solved by the robust value iteration algorithm  in Section \ref{sec:robustValueIter}, where the dual problem of the inner problem $w(s,k;0)$ is give by:
\begin{align*}
	w(s,k;0)=\max_{\lambda} & \bigg\{r(s,k)+\sum_{s'\in\mathcal{S}}\beta V(s^\prime,k)\bar{p}(s^\prime|s,k)+\lambda \sum_{s^\prime\in\mathcal{S}}\left(1-\bar{p}(s^\prime|s,k)\right) \\
	&+ \sum_{s^\prime\in\mathcal{S}}\left(\beta V(s^\prime,k)-\lambda\right)^+\left(\underline{p}(s'|s,k)-\bar{p}(s'|s,k)\right)\bigg\},
\end{align*}
 where \(\left(\beta V(s^\prime,k)-\lambda\right)^+=\max\{\beta V(s^\prime,k)-\lambda, 0\}\). Since this dual problem is piecewise linear  on $\lambda$ with break points $\beta V(s^\prime,k)$, $\forall s' \in \mathcal{S}$, the optimality is obtained at one of the break points. Suppose \(\lambda^\ast = \beta V(s'',k)\) for \(s''\in\mathcal{S}\). Then, the complementary slackness suggests \(p^\ast(s'|s,k)=\underline{p}(s'|s,k)\) if \(V(s',k)>V(s'',k)\),  \(p^\ast(s'|s,k)=\bar{p}(s'|s,k)\) if \(V(s',k)<V(s'',k)\), and \(p^\ast(s''|s,k)=1-\sum_{s'\in\mathcal{S}\backslash\{s''\}}p^\ast(s'|s,k)\). Corollary \ref{cor:worstcaseditr-IM}
follows when \(V(s',k)\) is non-increasing in \(s'\in\mathcal{S}\) for any given \(k\).

\subsection{Sensitivity Analysis } \label{sec:sensAnalysis}
Consider two problem instances (\(\Lambda_1\) and \(\Lambda_2\)) with ambiguity sets constructed  using the KL distance or the interval matrix model. Assume that these two problem instances satisfy the following: (1)  reward functions are the same, (2) the ambiguity set of problem \(\Lambda_1\) is contained in its counterpart of problem \(\Lambda_2\) (i.e, \(\mathcal{U}_{sk}^1 \subseteq \mathcal{U}_{sk}^2\)), and (3) all conditions that are needed to ensure the control-limit structure of the optimal robust policies are satisfied. Theorem \ref{thm:sensitivity} addresses the relationship between the optimal robust policies of these two systems. Let \(\zeta_{\text{rm}}^i(k)\) and \(\zeta_{\text{scrap}}^i(k)\) denote the remanufacturing and scrap limits of problem \(\Lambda_i,i=1,2,\) given \(k\), respectively. Let \(k_i^*\) be the same threshold defined in Theorems \ref{thm:Monotone-S-KL} and \ref{thm:Monotone-S-IM}.	We show that given \(k\), the remanufacturing thereshold in problem \(\Lambda_2\), \(\zeta_{\text{rm}}^2(k)\) is higher than or the same as its counter part in problem \(\Lambda_2\). In contrast, problem \(\Lambda_2\) has a higher scrap control limit given \(k\).  This shows that a decision maker needs to be  more conservative about initiating a remanufacturing process in anticipation of more transition uncertainties and that decision makers should consider scrapping early to receive the terminal rewards, to hedge against uncertainties in future operational gains. We summarize our results in the following theorem.
\begin{theorem}
	\label{thm:sensitivity}
	Let  \(\Lambda_1\) and \(\Lambda_2\) be two problem instances defined in this section. Then, the following holds.
	\begin{enumerate}
		\item [\textup{(a)}] for \(k<k_1^*\), \(\zeta^1_{\textup{rm}}(k) \leq \zeta^2_{\textup{rm}} (k)\) if $p_1^\ast(\tilde{s}|\tilde{s},k) = 0$, where \(\tilde{s} = \zeta^1_{\textup{rm}}(k)-1\).
		\item [\textup{(b)}] for \(k\ge k_1^*\), \(\zeta^1_{\textup{scrap}} (k)\geq \zeta^2_{\textup{scrap}} (k)\)
		\item [\textup{(c)}] $k_1^* \ge k_2^*$.  
	\end{enumerate}
\end{theorem}	
	
 Suppose that a decision maker has problem $\Lambda_1$ implemented and is aware of its optimal robust policy. Theorem \ref{thm:sensitivity} offers valuable insights on the optimal policy if the decision maker decides to be more conservative by considering a larger ambiguity set. It is worth noting that if the system does not allow self-transition, i.e., $\hat{p}(s|s,k)$ for all $s\in \mathcal{S}, k\in \mathcal{K}$, then Theorem \ref{thm:sensitivity} holds.

\section{Computational Study} \label{sec:computationalStudy}

\subsection{System Model Description}
Real-world turbofan engine operating data acquired from sensors are used to demonstrate our robust remanufacturing planning model. Procuring actual turbofan engine system fault progression data is typically time consuming and expensive. Hence, we use the data simulated using the Commercial Modular Aero-Propulsion System Simulation (C-MAPSS) software \citep{frederick2007user} developed at NASA to demonstrate our robust remanufacturing planning model and examine the performance of the optimal robust remanufacturing policies. The C-MAPSS engine is a 90,000 lb thrust class turbofan engine and has five rotating components: fan, low prpessure compressor(LPC), high pressure compressor(HPC), high pressure turbine(HPT), and low pressure turbine(LPT). The engine diagram in Figure \ref{fig:engine} shows the main elements of the engine model \citep{saxena2008damage}.

\begin{figure}[htbp]
	\centering
	\includegraphics[height=4cm]{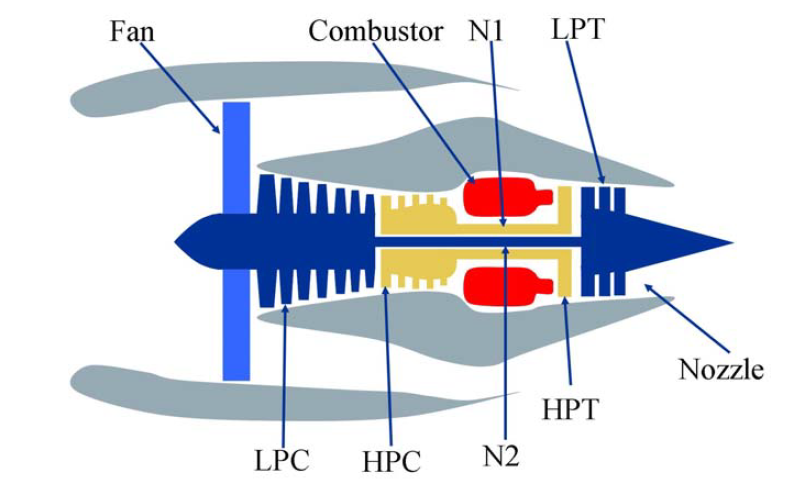}
	\caption{Simplified diagram of engine simulated in C-MAPSS \citep{saxena2008damage}}
	\label{fig:engine}
\end{figure}

The overall simulation is implemented in the MATLAB and Simulink environment, providing flexible interaction with the software user. C-MAPSS  offers  14 inputs and can produce several outputs for analysis. The system inputs include fuel flow, deviation from scheduled variable stator vanes angle, deviation from scheduled variable bleed valve position, and a set of 13 health parameters (e.g., Fan efficiency modifier, LPC flow modifier, HPC pressure-ratio modifier) that consist of flow, efficiency, and pressure-ratio modifiers for the fan, LPC, and HPC, and flow and efficiency modifiers for the HPT and LPT. Advanced users can readily modify and customize the model to their specific requirements to simulate the deterioration in any of the engine's five rotating components. The outputs include various sensor response surfaces and operability margins. The total of 21 outputs are summarized in Table \ref{tbl:CMAPSSOutput}.
\begin{table}[htbp]
	\centering
	\caption{C-MAPSS outputs \citep{saxena2008damage}}
	\begin{tabular}{ccc}
		\hline
		\multicolumn{1}{c}{Symbol} & Description                     & Units       \\ \hline
		T2                          & Total temperature at fan inlet  & $^{\circ}$R \\
		T24                         & Total temperature at LPC outlet & $^{\circ}$R \\
		T30                         & Total temperature at HPC outlet & $^{\circ}$R \\
		T50                         & Total temperature at LPT outlet & $^{\circ}$R \\
		P2                          & Pressure at fan inlet           & psia        \\
		P15                         & Total pressure in bypass-duct   & psia        \\
		P30                         & Total pressure at HPC outlet    & psia        \\
		Nf                          & Physical fan speed              & rpm         \\
		Nc                          & Physical core speed             & rpm         \\
		epr                         & Engine pressure ratio (P50/P2)  & -           \\
		Ps30                        & Static pressure at HPC outlet   & psia        \\
		phi                         & Ratio of fuel flow to Ps30      & pps/psi     \\
		NRf                         & Corrected fan speed             & rpm         \\
		NRc                         & Corrected core speed            & rpm         \\
		BPR                         & Bypass Ratio                    & -           \\
		farB                        & Burner fuel-air ratio           & -           \\
		htBleed                     & Bleed Enthalpy                  & -           \\
		Nf\_dmd                     & Demanded fan speed              & rpm         \\
		PCNfR\_dm                   & Demanded corrected fan speed    & rpm         \\
		W31                         & HPT coolant bleed               & lbm/s       \\
		W32                         & LPT coolant bleed               &   lbm/s          \\ \hline
	\end{tabular}
	\label{tbl:CMAPSSOutput}
\end{table}

\subsection{Dataset Description}
We consider the data pertaining to a single failure mode and a single operating condition. The dataset considered in this work consists of 100  units which are run to failure. Note that end-of-life can be subjectively determined as a function of operational thresholds that can be measured; these thresholds depend on user specifications to determine safe operational limits. For illustration purposes, we arbitrarily choose four features and plot the time series of these features for a randomly selected unit and all units (Figure \ref{fig:unitAllFeatures}).

\begin{figure}[htbp]
	\centering
		\subcaptionbox{Feature 7 \label{fig:unit6Feature01}}[0.24\linewidth]
	{
		\includegraphics[width=1\linewidth]{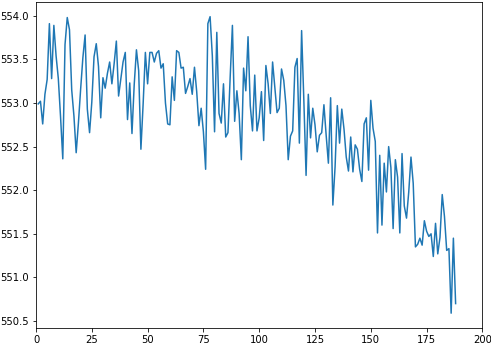}
	}
	\subcaptionbox{Feature 9 \label{fig:unit6Feature02}}[0.24\linewidth]
	{
		\includegraphics[width=1\linewidth]{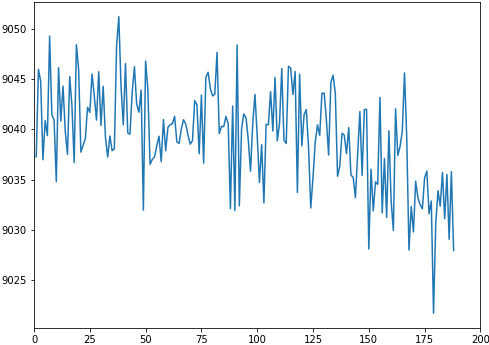}
	}
	\subcaptionbox{Feature 11 \label{fig:unit6Feature03}}[0.24\linewidth]
	{
		\includegraphics[width=1\linewidth]{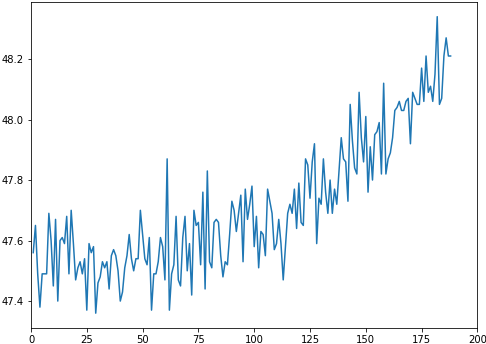}
	}
	\subcaptionbox{Feature 13 \label{fig:unit6Feature04}}[0.24\linewidth]
	{
		\includegraphics[width=1\linewidth]{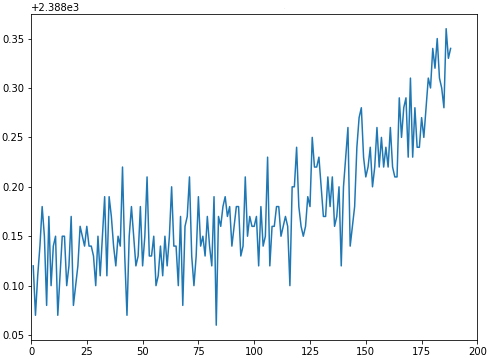}
	}
	\subcaptionbox{Feature 7 \label{fig:unitAllFeature01}}[0.24\linewidth]
	{
		\includegraphics[width=1\linewidth]{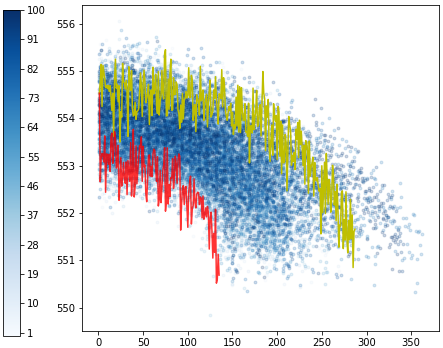}
	}
	\subcaptionbox{Feature 9 \label{fig:unitAllFeature02}}[0.24\linewidth]
	{
		\includegraphics[width=1\linewidth]{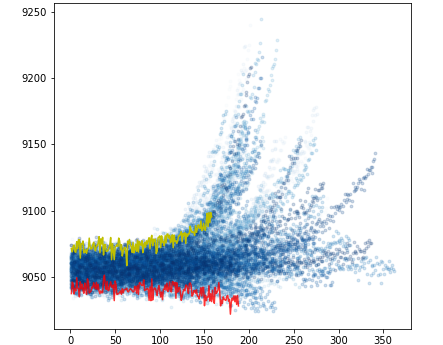}
	}
	\subcaptionbox{Feature 11 \label{fig:unitAllFeature03}}[0.24\linewidth]
	{
		\includegraphics[width=1\linewidth]{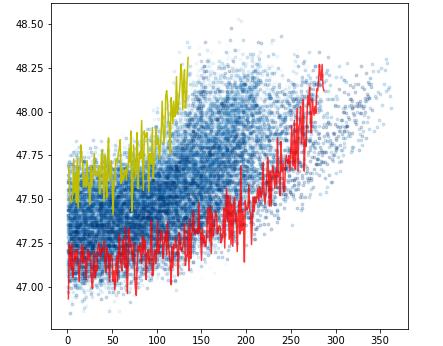}
	}
	\subcaptionbox{Feature 13 \label{fig:unitAllFeature04}}[0.24\linewidth]
	{
		\includegraphics[width=1\linewidth]{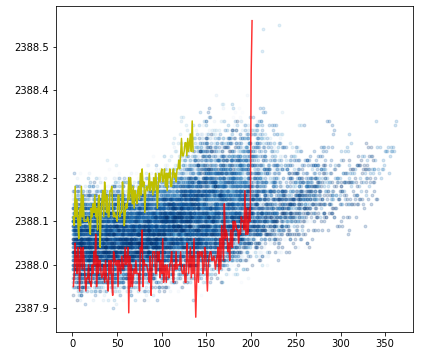}
	}
	\caption{Illustrations of raw sensor data sequences. (a)-(d),  time series of the selected features of unit 6. (e)-(h), time series of the selected features of all units. Solid lines are the time series of the unit that has the most maximum (yellow line) and minimum (red line) points. }
	\label{fig:unitAllFeatures}
\end{figure}

From Figure \ref{fig:unitAllFeatures}, we can see that the data contains a lot of noises. Various sources can contribute to noises, and the main sources of noise are manufacturing and assembly variations, process noise, and measurement noise to name a few important ones \citep{saxena2008damage}. Due to the large amount of noises and limited real-world operational data available, there often exists a high level of uncertainties in transition probabilities of the turbofan engines, and operators and manufacturers are in great need of robust remanufacturing planning.

\subsection{Parameter Estimation}
It is typically desirable to reduce the dimensionality of the data and reconstruct them from a lower dimensional samples. We therefore use the principal component analysis  method to compress the high-dimensional sensor outputs and use the first principle component that accounts for the largest variability of data as the system health indicator. We further discretize the obtained health indicator into 7 intervals, representing 7 condition states, as recommended by \cite{moghaddass2014integrated}.  The \(k\)-means method is used to partition health indicators. The nominal transition probability is estimated using the maximum likelihood method, i.e., \(\hat{p}(s'|s) =  \sum_{i=1}^m n_i(s'|s)/\sum_{i=1}^m \sum_{s'\in\mathcal{S}}n_i(s'|s)\), where \(n_i(s'|s)\) is the number of transitions from state \(s\) to \(s'\) for unit \(i\), and \(m\) is the total number of units in a sample. We construct the ambiguity sets as described in Section \ref{sec:construction}. Note that when constructing interval-matrix-based ambiguity sets, 30 bootstrap samples are used. Based on discussions with researchers and field engineers, an engine typically lose about 7\% useful life each time it is remanufactured. We modify the nominal transition probability matrix obtained for new turbonfan engines (i.e., \(k=0\)) to reflect such a loss for \(k> 0\). For all the following experiments, the nominal transitional probability matrices satisfy Assumption \ref{asm:transition}, and the lower and upper bounds of transition probabilities satisfy conditions \eqref{lowerboundConstraints-S} to \eqref{upperboundConstraints-K}.

\subsection{Experiments}
Next, we  demonstrate the structure of the optimal robust remanufacturing policy and examine the out-of-sample performance of the optimal robust policies. The following cost data is used for all experiments in this section: $g(s,k)=4-0.25s-0.25k, e(s,k)=1+0.25s+0.25k, c_r=2,$ and $c_s=0.5$. The discount factor $\beta$ is 0.9 for all following experiments.

\subsubsection{Policy Structures}\label{sec:violationTest}
We have established conditions to ensure control-limit policies for Model \eqref{Bellman} with ambiguity sets constructed using the KL distance or the interval matrix method. For illustration purposes, we show the structure of optimal robust policies for Model \eqref{Bellman} with KL-distance-based ambiguity sets.

As Figure \ref{fig:policyTheta} shows, the remanufacturing policies exhibit control-limit structure. We can also see that as \(\theta\) increases, the remanufacturing threshold \(\zeta_\text{rm}(k)\) increases and \(k^\ast\) decreases  (i.e., the  scrap is performed earlier). This implies that when  parameter uncertainty is large, a decision maker needs to be cautious about remanufacturing  used products and to consider scrapping at an earlier stage. This is because (1) the remanufacturing cost may not be offset by the subsequent operational gains due to large parameter uncertainties and (2) securing the fixed salvage value better hedges against uncertainties in future gains. 

\begin{figure}[htbp]
	\centering
	\subcaptionbox{\(\theta = 0\) \label{fig:policyTheta0}}[0.3\linewidth]
	{
		\includegraphics[width=1\linewidth]{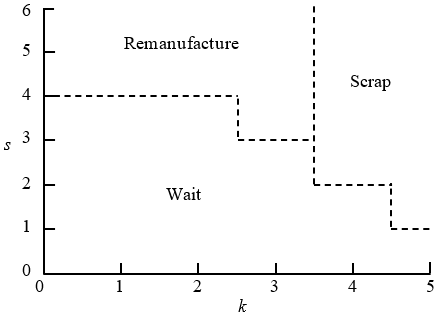}
	}
	\subcaptionbox{\(\theta = 0.5\) \label{fig:policyTheta0.3}}[0.3\linewidth]
	{
		\includegraphics[width=1\linewidth]{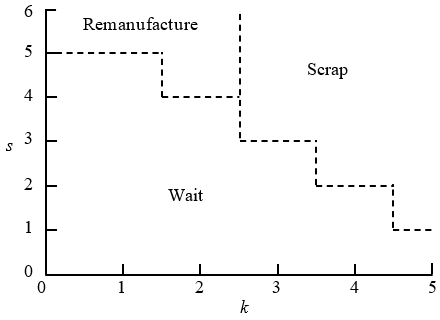}
	}	
	\subcaptionbox{\(\theta = 1\) \label{fig:policyTheta2.3}}[0.3\linewidth]
	{
		\includegraphics[width=1\linewidth]{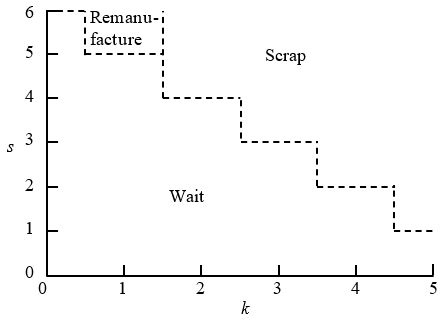}
	}
	\caption{Optimal robust policies for different \(\theta s\)}
	\label{fig:policyTheta}
\end{figure}

As stated earlier, the condition of Theorem \ref{thm:Monotone-K-KL}(a), which is the same as the condition of Theorem \ref{thm:Monotone-K-IM},  is restrictive and difficult to satisfy. We  further examine whether the optimal robust policies are still of control-limit type when this condition is violated. We  test a total of 5000 instances and the generation of the test instances is described in  Appendix B.1.  Out of the 3060 test instances that violate the condition of Theorem \ref{thm:Monotone-K-KL}(a), only  209  (i.e., approximately 6.8\%) instances violate the monotone structure. Therefore, we  believe that a control-limit policy with respect to \(k\) can be obtained in most practical cases even when the condition that guarantees it is violated.

\subsubsection{Impact of the Parameter Uncertianty }
We first conduct experiments to investigate the impact of the parameter uncertainty on the out-of-sample performance.  The radius \(\theta\) determines the size of the KL-distance-based ambiguity set and the confidence level \(\alpha\) determines the size of the ambiguity set constructed using bootstrap resampling. For notational convenience, we use \(\psi\) to denote the hyperparameter that controls the size of the ambiguity set.

We use a training dataset  \(\mathcal{N}\). The optimal robust policies of Model \eqref{Bellman} with ambiguity sets constructed under different hyperparameter values using the training dataset, \(\pi_\mathcal{N}(\psi)\), are then implemented in a  test dataset   \(\mathcal{M}\) to assess the out-of-sample performance. We examine two  performance measurements: the average reward and the reliability of performance guarantees. The average reward is defined as  \(\bar{\nu}_\mathcal{N}(\psi) = \sum_{i\in\mathcal{M}}\nu_i(\pi_\mathcal{N}(\psi))/|\mathcal{M}|\), where \(\nu_i(\pi_\mathcal{N}(\psi))\) is the expected reward of robust policy \(\pi_\mathcal{N}(\psi)\) for test sample \(i\in \mathcal{M}\) when the system is brand new (\(s=0,k=0\)). The reliability is defined as the probability of the event \(\bar{\nu}_\mathcal{N}(\psi)\ge V_\mathcal{N}(\psi)\), where \(V_\mathcal{N}(\psi)\) is the in-sample value of \(V(0,0)\) under  \(\psi\).

Figure \ref{fig:withoutValidPerformance} depicts the experiment results when the sizes of training dataset is 5 (\(|\mathcal{N}| = 5\)) and size of the test dataset is 50 (\(|\mathcal{M}| = 50\)).
From Figure \ref{fig:withoutValidPerformance}(\subref{fig:withoutValidRewardKL}), we observe that the average reward of the robust policy is slightly higher than that of the nominal policy when \(\theta\) is not too large. As   \(\theta\) keeps increasing, the average reward of the robust policy deteriorates because the robust policy is too conservative.  The empirical reliability visualized in Figure \ref{fig:withoutValidPerformance}(\subref{fig:withoutValidReliabilityKL})  is in general non-decreasing in  \(\theta\), and the reliability of the performance guarantee under the robust approach is much higher than that under the nominal approach.   We also find that the out-of-sample average reward using a robust approach is better as long as the reliability of the performance guarantee is noticeably smaller than 1 and deteriorates when it is close to 1. Figure \ref{fig:withoutValidPerformance}(c) and (d) present the out-of-sample performance and the reliability of Model \eqref{Bellman} with the interval-matrix-based ambiguity sets, respectively. Similar patterns are observed. Results of this experiment provide an empirical justification of adopting a robust remanufacturing approach, especially when the size of the dataset is small.

\begin{figure}[htbp]
	\centering
	\subcaptionbox{ \label{fig:withoutValidRewardKL}}[0.35\linewidth]
	{
		\includegraphics[width=1\linewidth]{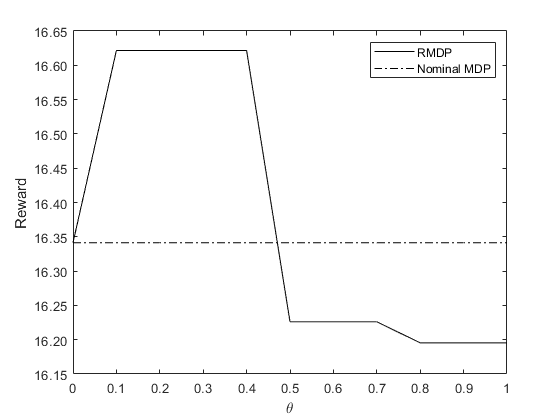}
	}
	\subcaptionbox{\label{fig:withoutValidReliabilityKL}}[0.35\linewidth]
	{
		\includegraphics[width=1\linewidth]{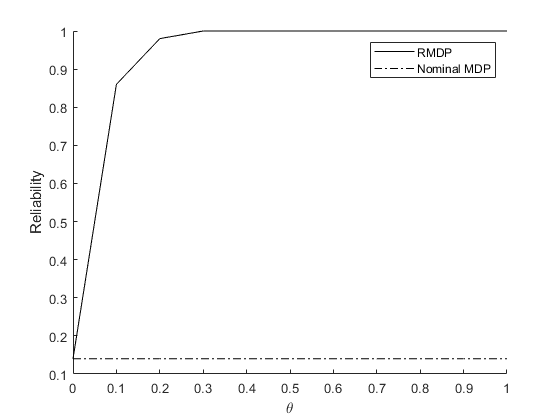}
	}
	
	\subcaptionbox{
		\label{fig:withoutValidRewardIM}}[0.35\linewidth]
	{
		\includegraphics[width=1\linewidth]{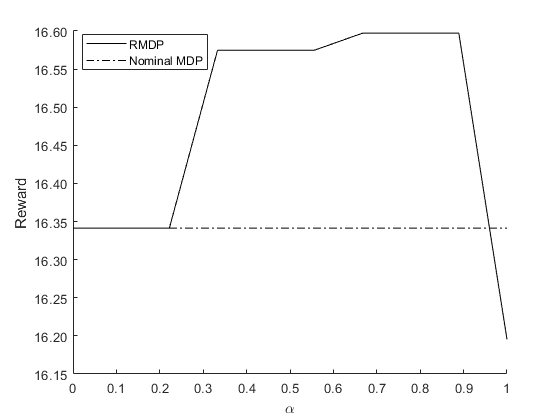}
	}
	\subcaptionbox{ \label{fig:withoutValidReliabilityIM}}[0.35\linewidth]
	{
		\includegraphics[width=1\linewidth]{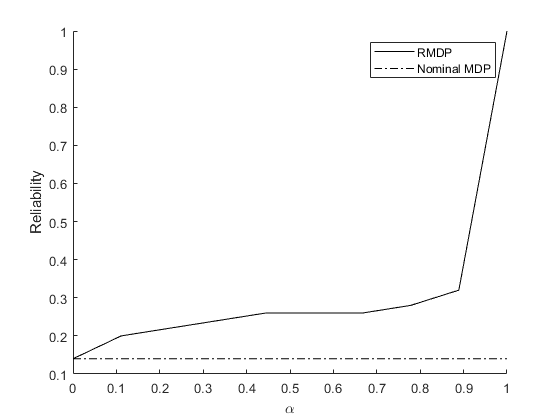}
	}
	
	\caption{Out-of-sample reward \(\bar{\nu}_\mathcal{N}(\psi)\) and reliability \(\text{Pr}\{\bar{\nu}_\mathcal{N}(\psi) \geq V_\mathcal{N}(\psi)\}\) as a function of \(\psi\).    (a)-(b)  KL-distance-based ambiguity set. (c)-(d)  Interval-matrix-model-based ambiguity set.}
	\label{fig:withoutValidPerformance}
\end{figure}

\subsubsection{Remanufacturing Planning Driven by Out-of-Sample Performance}

From the previous experiment on the impact of the parameter uncertainty, it is shown that different hyperparameter \(\psi\) values may lead to robust remanufacturing policies with different out-of-sample performance \(\bar{\nu}_\mathcal{N}(\psi)\). It is desired to select a \(\psi\)   that maximizes the  average award \(\bar{\nu}_\mathcal{N}(\psi)\). This, however, requires the true transition probability that is not precisely known. We select the optimal \(\psi\) via validation using the training data. Specifically, we randomly select 60\% of the training dataset \(\mathcal{N}\) for training and the remaining 40\% of the training data is used for validation. Using newly formed training dataset to construct the ambiguity sets,  solve  Model \eqref{Bellman} for a finite number of candidate hyperparameter \(\psi\). Use the validation dataset to evaluate the out-of-sample performance of $\pi_{\mathcal{N}}(\psi)$, select the optimal $\psi^\ast$ as the one that maximizes $\bar{\nu}_\mathcal{N}(\psi)$ of the validation set,  and report $\pi_{\mathcal{N}}(\psi^\ast)$ as the data-driven solution.


Figure \ref{fig:validPerformance}(a) shows the mean value of the out-of-sample performance \(\bar{\nu}_\mathcal{N}(\psi^\ast)\) as a function of the sample size \(|\mathcal{N}|\).  We also observe that both out-of-sample and in-sample performances exhibit asymptotic consistency. Figure \ref{fig:validPerformance}(b) shows the mean of the reliability of the guaranteed performance under different sample sizes. We can see that the robust policy significantly outperforms the nominal one, particularly when the training data is scarce. As more data become available, the optimal robust policy converges to the nominal policy, and so does the performance of the robust  policy.  Figure \ref{fig:validPerformance}(c) reports in-sample estimate \(V_{\mathcal{N}}(\psi)\). We can see that the nominal approach is over-optimistic while the robust approaches act on the cautious side. 


\begin{figure}[htbp]
	\centering
	\subcaptionbox{ \label{fig:validRewardKL}}[0.3\linewidth]
	{
		\includegraphics[width=1\linewidth]{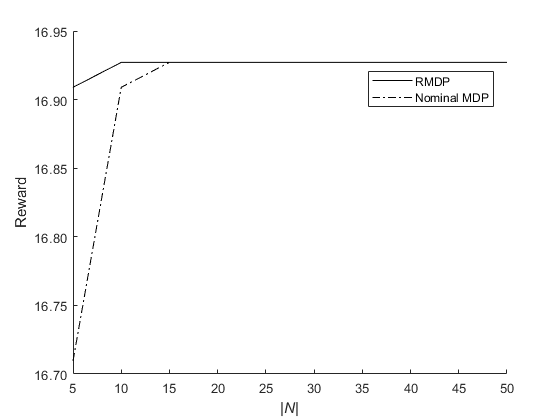}
	}
	\subcaptionbox{ \label{fig:validReliabilityKL}}[0.3\linewidth]
	{
		\includegraphics[width=1\linewidth]{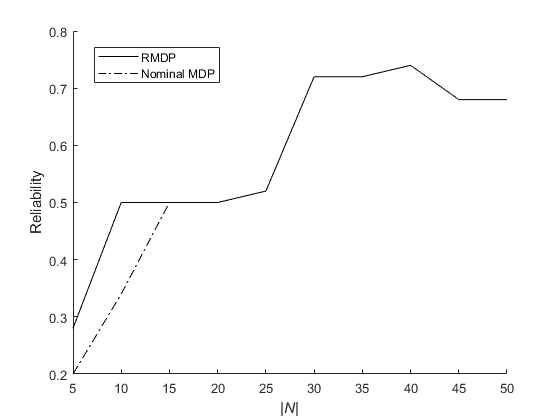}
	}
	\subcaptionbox{ \label{fig:validCerKL}}[0.3\linewidth]
	{
		\includegraphics[width=1\linewidth]{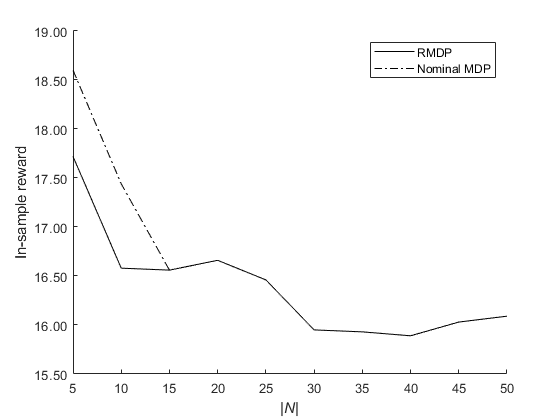}
	}	
	
	\subcaptionbox{ \label{fig:validRewardIM}}[0.3\linewidth]
	{
		\includegraphics[width=1\linewidth]{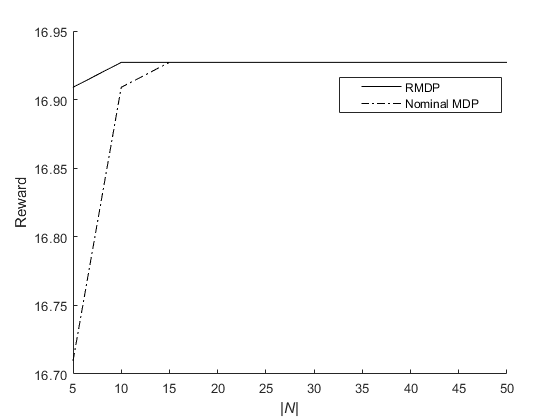}
	}
	\subcaptionbox{ \label{fig:validReliabilityIM}}[0.3\linewidth]
	{
		\includegraphics[width=1\linewidth]{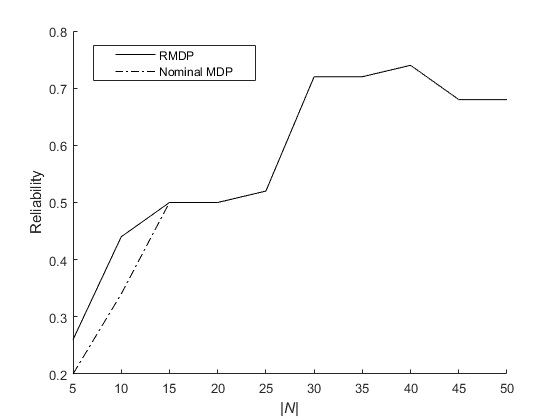}
	}
	\subcaptionbox{ \label{fig:validCerIM}}[0.3\linewidth]
	{
		\includegraphics[width=1\linewidth]{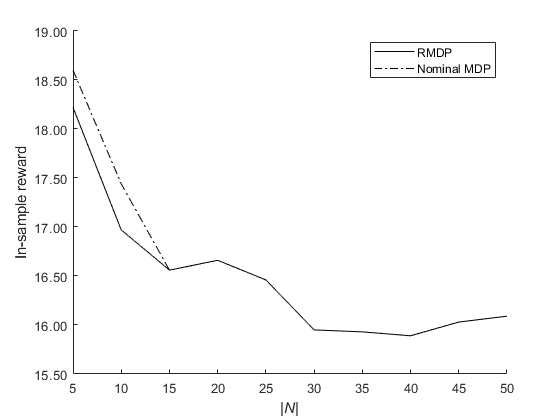}
	}	
	\caption{ Out-of-sample reward \(\bar{\nu}_\mathcal{N}(\psi^\ast)\), reliability \(\text{Pr}\{\bar{\nu}_\mathcal{N}(\psi^\ast) \geq V_\mathcal{N}(\psi^\ast)\}\), and  in-sample reward \(V_\mathcal{N}(\psi^\ast)\) as a function of \(|\mathcal{N}|\).  (a)-(c) KL-distance-based ambiguity set. (d)-(f) Interval-matrix-model-based ambiguity set.}
	\label{fig:validPerformance}
\end{figure}

\subsubsection{Remanufacturing Planning Driven by Reliability}
\begin{figure}[t!]
	\centering
	\subcaptionbox{ \label{fig:validReRewardKL}}[0.3\linewidth]
	{
		\includegraphics[width=1\linewidth]{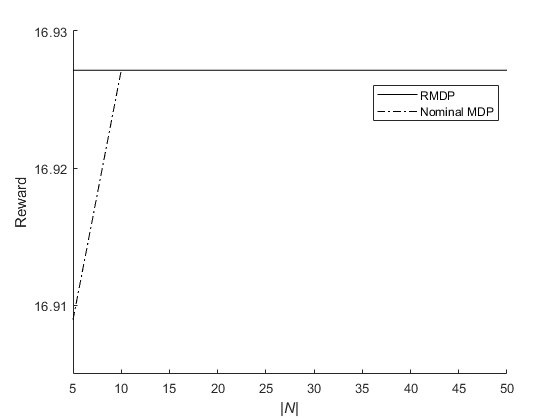}
	}
	\subcaptionbox{ \label{fig:validReReliabilityKL}}[0.3\linewidth]
	{
		\includegraphics[width=1\linewidth]{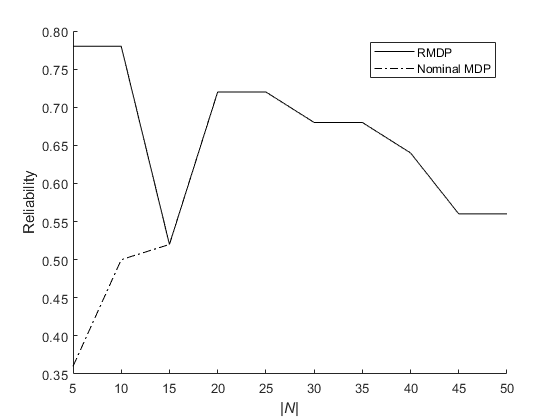}
	}
	\subcaptionbox{ \label{fig:validReCerKL}}[0.3\linewidth]
	{
		\includegraphics[width=1\linewidth]{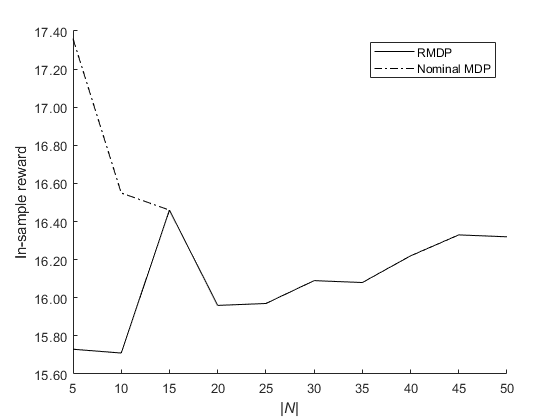}
	}	
	
	\subcaptionbox{ \label{fig:validReRewardIM}}[0.3\linewidth]
	{
		\includegraphics[width=1\linewidth]{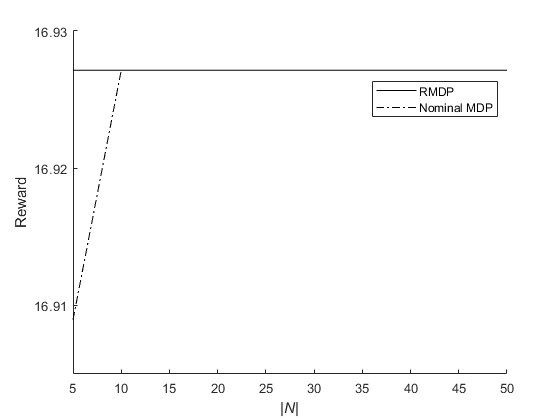}
	}
	\subcaptionbox{ \label{fig:validReReliabilityIM}}[0.3\linewidth]
	{
		\includegraphics[width=1\linewidth]{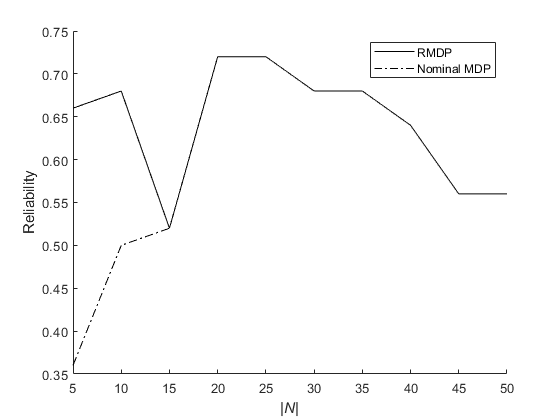}
	}
	\subcaptionbox{ \label{fig:validReCerIM}}[0.3\linewidth]
	{
		\includegraphics[width=1\linewidth]{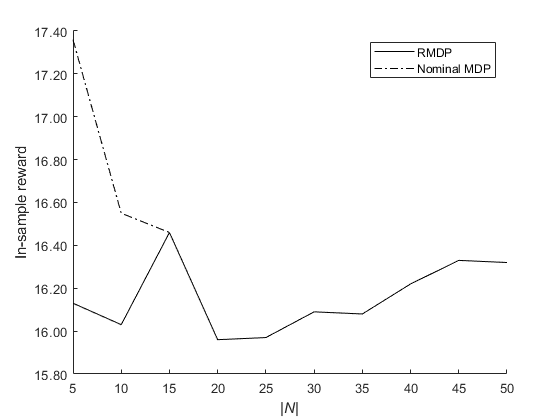}
	}	
	\caption{ Out-of-sample reward \(\bar{\nu}_\mathcal{N}(\psi_\gamma)\), reliability \(\text{Pr}\{\bar{\nu}_\mathcal{N}(\psi_\gamma) \geq V_\mathcal{N}(\psi_\gamma)\}\), and in-sample reward \(V_\mathcal{N}(\psi_\gamma)\) as a function of \(|\mathcal{N}|\) with a reliability guarantee of \(\gamma\).  (a)-(c) KL-distance-based ambiguity set. (d)-(f) Interval-matrix-model-based ambiguity set.}
	\label{fig:validReliability}
\end{figure}

From the previous experiment, we can see that there exists some trade-off between the out-of-sample performance and the reliability of the performance guarantee; reliability may be sacrificed if the optimal hyperparameter \(\psi\) is selected by only maximizing the out-of-sample performance. In the next experiment, we consider an alternative objective that chooses a hyperparameter \(\psi\) that  results in satisfactory out-of-sample performance while ensuring a prescribed reliability level. We use the method described in \cite{Kuhn2018Wasserstein} to find the smallest \(\psi\)  for which a desired reliability level (\(\gamma\)) is guaranteed. An estimator of  \(\psi_{\gamma}\) is constructed via bootstrapping the training data as follows. Construct \(q\) bootstrap samples  (with replacement) from the original training dataset. Solve Model \eqref{Bellman} for a finite number of hyperparameters for each bootstrap sample (\(\mathcal{N}_q\)) and obtain the optimal value \(V_{\mathcal{N}_q}(\psi)\), and then estimate the out-of-sample performance for the corresponding validation dataset. Set \(\psi_\gamma\) to the smallest \(\psi\) that leads to a reliability level of \(\gamma\), that is, the out-of-sample performance \(\bar{\nu}_{\mathcal{N}_q}(\psi_\gamma)\) for the validation sets exceeds the in-sample performance \(V_{\mathcal{N}_q}(\psi_\gamma)\) in at least \(\gamma \times q\) different bootstrap samples. Resolve Model \eqref{Bellman} for \(\psi_\gamma\) and obtain the data-driven remanufacturing policies using the original training dataset \(\mathcal{N}\). The ambiguity sets constructed in this method calibrates their size to guarantee the desired reliability level \(\gamma\).

Figure \ref{fig:validReliability} depicts the result when \(q=30\) and \(\gamma=0.7\). From Figures \ref{fig:validReliability}(\subref{fig:validReReliabilityKL}) and (\subref{fig:validReReliabilityIM}), we can see that the empirical reliabilities obtained from the robust approach are close to the desired reliability target and exceed the prescribed target in many cases.

\section{Conclusion and Future Work}\label{sec:conclusion}
In this paper, we consider the problem of remanufacturing planning in the presence of parameter uncertainty. We formulate the problem as a robust Markov decision process in which the true transition probability is unknown but lies in an ambiguity set with high confidence. Two statistical methods are used to construct the ambiguity set: the KL distance, and  bootstrap resampling. We investigate the structure of the optimal robust policies and establish conditions to ensure the policies are of control-limit type. We also establish sufficient conditions for some of the intuitive results seen in our computational study. In particular, we derive the general decision insights for two systems\textemdash one system's ambiguity set is contained in the other's; we show that when large uncertainty in transition dynamics presents, the decision maker needs to be cautious about remanufacturing a product and should consider early scrapping to hedge against future uncertainties. We demonstrate the structure of the optimal robust policies via a computational study using the simulated operational data of the turbofan engine operated by NASA, investigate the  out-of-sample performance, and derive the data-driven solutions to improve the out-of-sample performance. 

Future extensions of this work will focus on investigating optimal production planning and inventory control policies for remanufacturing that build  on this work. Moreover, at each decision epoch, decision makers make new observation about the system, and an important question that arises is that how  the information that becomes available in the decision process can be leveraged to resolve some ambiguity, so that the optimal robust policies are not overly conservative. In addition, an implicit assumption made in this paper is that the states of a system are directly observable (i.e., the  sensor data reveal the underlying state of the system with certainty). In practice, many systems are not directly observable and the states have to be inferred from signals collected. Future work will investigate the partially observable Markov decision process with parameter uncertainty.

\bibliography{literature}
\section*{Appendix}	
\setcounter{equation}{0}
\renewcommand{\theequation}{A.\arabic{equation}}
\setcounter{lemma}{0}
\renewcommand{\thelemma}{A.\arabic{lemma}}
\setcounter{corollary}{0}
\renewcommand{\thecorollary}{A.\arabic{corollary}}

\subsection*{A.1 Proof of Proposition \ref{prop:reformulation-KL}}
The value function defined in \eqref{Bellman} involves solving an inner problem for any given $s \in \mathcal{S}$ and $k \in \mathcal{K}$ as follows
\begin{align}
w(s,k;0) = &
\min    		   \ r(s,k) + \beta \sum_{s' \in \mathcal{S}}p(s'|s,k)V(s',k) \nonumber \\
\ \text{s.t.}\quad  & \ \sum_{s' \in \mathcal{S}} p(s'|s,k) = 1,\ \sum_{s' \in \mathcal{S}} \log\left( \dfrac{p(s'|s,k)}{\hat{p}(s'|s,k)} \right) p(s'|s,k) \leq \theta \label{KL_Reform} \\
\quad & ~~ p(s'|s,k) \geq 0, s' \in \mathcal{S}. \nonumber 
\end{align}

\noindent
The Lagrangian dual problem of \eqref{KL_Reform} is
\begin{align*}
\max_{\lambda \text{ free}, \mu \ge 0} \ L(\lambda,\mu) ~~ 
\text{s.t.} \ L(\lambda,\mu) = \min_{\pmb{p}(\cdot|s,k) \ge 0} L(\lambda,\mu,\pmb{p}(\cdot|s,k))
\end{align*}

\noindent
where the Lagrangian function is 
\begin{align*}
L(\lambda,\mu,\pmb{p}(\cdot|s,k)) & = r(s,k)+ \beta \sum_{s' \in \mathcal{S}} p(s'|s,k) V(s',k) + \lambda \left(1-\sum_{s' \in \mathcal{S}} p(s'|s,k)\right) \nonumber \\
& \hspace{0.5cm} + \mu \left(\sum_{s' \in \mathcal{S}} p(s'|s,k)  \log\left(\frac{p(s'|s,k)}{\hat{p}(s'|s,k)}\right) - \theta\right) \\
& = r(s,k) + \lambda - \mu \theta + \sum_{s' \in \mathcal{S}} \left( \beta V(s',k) - \lambda + \mu \log\left(\frac{p(s'|s,k)}{\hat{p}(s'|s,k)}\right) \right) p(s'|s,k).
\end{align*}

\noindent
The strong duality holds because $\hat{\pmb{p}}(\cdot|s,k)$ is a strictly feasible solution to the problem \eqref{KL_Reform} and the Slater condition holds. The first order conditions of the Lagrangian function give
\begin{align}
\frac{\partial L(\lambda,\mu,\pmb{p}(\cdot|s,k))}{\partial p(s'|s,k)} & = \beta V(s',k) - \lambda + \mu \log\left(\frac{p(s'|s,k)}{\hat{p}(s'|s,k)}\right) + \mu = 0, \ \forall s' \in \mathcal{S} \nonumber \\
\Rightarrow \hspace{1.5cm} p(s'|s,k) & = \hat{p}(s'|s,k) \exp \left( \dfrac{- \beta V(s',k) + \lambda - \mu}{\mu} \right), \ \forall s' \in \mathcal{S}. \label{eq:KL_pk}
\end{align}

\noindent
By substituting \eqref{eq:KL_pk} into the Lagrangian function, the dual problem becomes
\begin{equation*}
\max_{\lambda \text{ free}, \mu \ge 0} L(\lambda,\mu) = r(s,k) + \lambda - \mu \theta - \exp\left(\dfrac{\lambda-\mu}{\mu} \right) \mu \sum_{s' \in \mathcal{S}} \hat{p}(s'|s,k) \exp(\dfrac{- \beta V(s',k)}{\mu}).
\end{equation*}

\noindent
Again, the first order conditions give
\begin{align}
\frac{\partial L(\lambda,\mu)}{\partial \lambda} & = 1 - \exp\left(\frac{\lambda-\mu}{\mu} \right) \sum_{s' \in \mathcal{S}} \hat{p}(s'|s,k) \exp\left(\frac{- \beta V(s',k)}{\mu}\right) = 0 \nonumber \\
\Rightarrow \hspace{1.5cm} \lambda &= -\mu \log \left( \sum_{s' \in \mathcal{S}}  \hat{p}(s'|s,k) \exp\left( \frac{- \beta V(s',k)}{\mu} \right) \right) + \mu. \label{eq:lambda}
\end{align}

\noindent
The dual problem can be rewritten as
\begin{align*}
\max_{\mu \ge 0} & \ L(\mu) = r(s,k) - \mu \log \left( \sum_{s' \in \mathcal{S}}  \hat{p}(s'|s,k) \exp\left( \frac{- \beta V(s',k)}{\mu} \right) \right) - \mu \theta.
\end{align*} 

\noindent
By combining \eqref{eq:KL_pk} and \eqref{eq:lambda}, we have the worst-case transitional probabilities as
\begin{align*}
p^*(s'|s,k) = \dfrac{\hat{p}(s'|s,k) \exp\left(- \beta V(s',k) / \mu^*_{sk} \right) }{\sum_{s'' \in \mathcal{S}}\hat{p}(s''|s,k) \exp\left(- \beta V(s'',k) / \mu^*_{sk} \right)}, \ \forall s' \in \mathcal{S}.
\end{align*}

\noindent
where $\mu^*_{sk}$ is the optimal solution of the dual problem with given $s$ and $k$.

\subsection*{A.2 Proof of Proosition \ref{prop:Monotone-ValueFunction-KL}}
Let $V^n(s,k)= \max_{a \in \mathcal{A}} w^{n}(s,k;a)$ denote the value function at the $n$th iteration of the robust value iteration algorithm  in Section \ref{sec:robustValueIter}. We will show that $V^n(s,k)$ is non-increasing in $s \in \mathcal{S}$ and $k\in\mathcal{K}$ for any integer $n \ge 0$ by induction. Then, the theorem follows because the robust value iteration algorithm converges to an optimal policy.

We set the initial value as $V^0(s,k)=0$ for all $s\in\mathcal{S}$ and $k\in\mathcal{K}$. First,  we show that \(V(s,k)\) is non-increasing in \(s\in \mathcal{S}\) for all \(k\in \mathcal{K}\). Because \(V^0(s,k) = 0\) for all $s\in\mathcal{S}$, the induction holds at the initial iteration. Assume that $V^n(s,k)$ is non-increasing in $s \in \mathcal{S}$ for $n = 1,\ldots,m-1$. Let $s^\prime, s \in \mathcal{S}$ with $s' > s$ and  \(\mu_{sk}^*\) be the optimal solution of the dual problem \eqref{reformulation-KL-a0} defined in Theorem \ref{prop:reformulation-KL} for any give state \((s,k) \in \mathcal{S}\times \mathcal{K}\). We consider two cases at iteration \(m\). If \(a=0\), we have

\begin{align}
w^{m}(s^\prime,k; 0)& = \max_{\mu > 0} \ r(s^\prime,k) -\mu \log \left( \sum_{s'' \in \mathcal{S}}  \hat{p}(s''|s',k) \exp\left( \frac{- \beta V^{m-1}(s'',k)}{\mu} \right) \right) -\mu\theta \nonumber \\
& = r(s^\prime,k) -\mu_{s'k}^* \log \left( \sum_{s'' \in \mathcal{S}} \hat{p}(s''|s',k) \exp\left( \frac{- \beta V^{m-1}(s'',k)}{\mu_{s'k}^*} \right) \right) -\mu_{s'k}^* \theta \nonumber \\
& \le r(s,k) -\mu_{s'k}^* \log \left( \sum_{s'' \in \mathcal{S}} \hat{p}(s''|s',k) \exp\left( \frac{- \beta V^{m-1}(s'',k)}{\mu_{s'k}^*} \right) \right) -\mu_{s'k}^*\theta \label{eq:mono10}\\
& \le r(s,k) -\mu_{s'k}^* \log \left( \sum_{s'' \in \mathcal{S}} \hat{p}(s''|s,k) \exp\left( \frac{- \beta V^{m-1}(s'',k)}{\mu_{s'k}^*} \right) \right) -\mu_{s'k}^* \theta \label{eq:mono11} \\
&\le \max_{\mu > 0} \ r(s,k) -\mu \log \left( \sum_{s'' \in \mathcal{S}} \hat{p}(s''|s,k) \exp\left( \frac{- \beta V^{m-1}(s'',k)}{\mu} \right) \right) -\mu\theta \nonumber\\
&= w^{m}(s,k; 0) \nonumber
\end{align}

\noindent
The inequality \eqref{eq:mono10} holds because \(r(s',k) \le r(s,k)\). The inequality \eqref{eq:mono11} follows  Lemma 4.7.2 in \cite{puterman2014markov} because \(\pmb{P}(\cdot|\cdot,k)\) is IFR and \(V^{m-1}(s,k)\) is non-increasing in \(s\) given \(k\) by the induction hypothesis.

If \(a = 1\), we have \(w^m(s,k;1)  =w^m(s^\prime,k;1)= -c_\text{r}+\beta V^{m-1}(0,k+1)\). Therefore, \(w^m(s,k;1)\) is non-increasing in \(s\) given \(k\). Similarly, since \(w^m(s,k;2)=w^m(s^\prime,k;2)=c_\text{s}\), \(w^m(s,k;2)\) is also non-increasing in \(s\) given \(k\). Since \(V^m(s,k)=\max_{a\in\mathcal{A}} w^m(s,k;a)\) \(\geq \max_{a\in\mathcal{A}} w^m(s^\prime,k;a) = V^m(s^\prime,k)\), the induction hypothesis holds at iteration \(m\). 

Next,  we show that \(V(s,k)\) is non-increasing in \(k \in \mathcal{K}\) for all  \(s \in \mathcal{S}\). Because \(V^0(s,k) = 0\) for all $k\in\mathcal{K}$, the induction holds at the initial iteration. Assume for any given $s\in\mathcal{S}$, $V^n(s,k)$ is non-increasing in $k \in \mathcal{K}$  for $n = 0,\ldots,m-1$. We consider two cases at iteration \(m\). If \(a=0\), we have

\begin{align}
&\ w^{m}(s,k+1; 0) \nonumber \\
= &\ \max_{\mu > 0} \ r(s,k+1) -\mu \log \left( \sum_{s' \in \mathcal{S}}  \hat{p}(s''|s,k+1) \exp\left( \frac{- \beta V^{m-1}(s',k+1)}{\mu} \right) \right) -\mu\theta \nonumber \\
= &\ r(s,k+1) -\mu_{s,k+1}^* \log \left( \sum_{s' \in \mathcal{S}}  \hat{p}(s'|s,k+1) \exp\left( \frac{- \beta V^{m-1}(s',k+1)}{\mu_{s,k+1}^*} \right) \right) -\mu_{s,k+1}^* \theta \nonumber \\
\le&\ r(s,k) -\mu_{s,k+1}^* \log \left( \sum_{s' \in \mathcal{S}}  \hat{p}(s'|s,k+1) \exp\left( \frac{- \beta  V^{m-1}(s',k)}{\mu_{s,k+1}^*} \right) \right) -\mu_{s,k+1}^*\theta \label{eq:mono20}\\
\le&\ r(s,k) -\mu_{s,k+1}^* \log \left( \sum_{s' \in \mathcal{S}}  \hat{p}(s'|s,k) \exp\left( \frac{- \beta V^{m-1}(s',k)}{\mu_{s,k+1}^*} \right) \right) - \mu_{s,k+1}^*\theta \label{eq:mono21} \\
\le&\ \max_{\mu > 0} \ r(s,k) -\mu \log \left( \sum_{s' \in \mathcal{S}}  \hat{p}(s''|s,k) \exp\left( \frac{- \beta V^{m-1}(s',k)}{\mu} \right) \right) -\mu\theta \nonumber\\
=&\ w^{m}(s,k; 0) \nonumber 
\end{align}

\noindent
The inequality \eqref{eq:mono20} holds because \(r(s,k+1) \le r(s,k)\) and $V^{m-1}(s,k+1) \le V^{m-1}(s,k)$ by the induction hypothesis. The Inequality \eqref{eq:mono21} follows  Lemma 4.7.2 in \cite{puterman2014markov} because \(\pmb{P}(\cdot|\cdot,k+1)\succeq \pmb{P}(\cdot|\cdot,k)\) by Assumption \ref{asm:transition}(b)  and \(V^{m-1}(s,k)\) is non-increasing in $s\in\mathcal{S}$.

If \(a = 1\), we have \(w^m(s,k;1) = -c_\text{r}+\beta V^{m-1}(0,k+1) \ge -c_\text{r}+\beta V^{m-1}(0,k+2)=w^m(s,k+1;1)\). Therefore, \(w^m(s,k;1)\) is non-increasing in \(k\) for all \(s \in \mathcal{S}\). Similarly, since \(w^m(s,k;2)=w^m(s,k+1;2)=c_\text{s}\), \(w^m(s,k;2)\) is also non-increasing in \(k\) for all \(s \in \mathcal{S}\). Since \(V^m(s,k)=\max_{a\in\mathcal{A}} w^m(s,k;a)\) \(\geq \max_{a\in\mathcal{A}} w^m(s,k+1;a) = V^m(s,k+1)\), the induction hypothesis holds at iteration \(m\).

\subsection*{A.3 Proof of Theorem \ref{thm:Monotone-S-KL}}
%

We first show that the optimal policy is of  control-limit type for all \(k \in \mathcal{K}\). Let \(s'>s\). We consider two cases: (\textit{i}) If \(a^\ast(s,k)=1\), then \(V(s,k)=w(s,k;1) = -c_\text{r}+\beta V(0,k+1) = w(s',k;1) \le V(s',k)\). Because \(V(s,k)\) is non-increasing in \(s\) for all \(k \in \mathcal{K}\), \(V(s,k)\ge V(s',k)\). Thus, we have \( V(s^\prime,k)=w(s^\prime,k;1)\) and  \(a^\ast(s^\prime,k) = 1\). (\textit{ii}) If \(a^\ast(s,k)=2\), then \(V(s,k)=w(s,k;2)=c_\text{s}=w(s',k;2)\), and by Theorem \ref{prop:Monotone-ValueFunction-KL}, \(V(s,k)\ge V(s',k)\), we have \(V(s',k)=w(s',k;2)\) and \(a^\ast(s',k)=2\).

Next, we show the existence of the threshold \(k^\ast\). This is equivalent to show that if \(a^\ast(s,k) = 2\) for some \(k\), then \(a^\ast(s,k+1) = 2\). Since \(V(s,k)=w(s,k;2)=c_\text{s}=w(s,k+1;2)\le V(s,k+1)\) and \(V(s,k)\ge V(s,k+1)\), we have \(V(s,k+1)=w(s,k+1;2)\) and hence \(a^\ast(s,k+1) = 2\).

\subsection*{A.4 Proof of Theorem \ref{thm:Monotone-K-KL}}

We first prove that \(\zeta_{\text{rm}}(k)\) is non-increasing in \(k,\ \forall k \in \{0,\ldots, k^*-1\}\). This is equivalent to show that \(a^\ast(s,k+1) = 1\) if \(a^\ast(s,k) = 1\ \forall k \in \{0,\ldots, k^*-2\}\). We prove this by contradiction.
Suppose \(a^\ast(s,k)=1\) but \(a^\ast(s,k+1)=0\) for some \(s \in \mathcal{S}\) and \(k \in \{0,\ldots, k^*-2\}\). Then, we have \(w(s,k;1) \geq w(s,k;0)\), \(w(s,k+1;1) < w(s,k+1;0)\) and hence,
\begin{equation}
\label{eg:mono-k-kl-1}
w(s,k;1) - w(s,k+1;1) > w(s,k;0) - w(s,k+1;0).
\end{equation}

\noindent
The right hand side (RHS) of Equation \eqref{eg:mono-k-kl-1}  be rewritten as 
\begin{align}
\textnormal{RHS} &= r(s,k) + \max_{\mu> 0}\left(-\mu\log\left(\sum_{s'\in\mathcal{S}}\hat{p}(s'|s,k) \exp\left(\dfrac{-\beta V(s',k)}{\mu}\right)\right) -\mu\theta\right) \nonumber\\ 
& \hspace{.5cm} - r(s,k+1) - \max_{\mu> 0}\left(-\mu\log\left(\sum_{s'\in\mathcal{S}}\hat{p}(s'|s,k+1) \exp\left(\dfrac{-\beta V(s',k+1)}{\mu}\right)\right) -\mu\theta\right) \nonumber\\
&\geq r(s,k) + \left(-\mu_{s,k+1}^* \log\left(\sum_{s'\in\mathcal{S}}\hat{p}(s'|s,k) \exp\left(\dfrac{-\beta V(s',k)}{\mu_{s,k+1}^*}\right)\right) -\mu_{s,k+1}^*\theta\right) \nonumber\\ 
& \hspace{.5cm} - r(s,k+1) - \left(-\mu_{s,k+1}^*\log\left(\sum_{s'\in\mathcal{S}}\hat{p}(s'|s,k+1) \exp\left(\dfrac{-\beta V(s',k+1)}{\mu_{s,k+1}^*}\right)\right) -\mu_{s,k+1}^*\theta\right) \nonumber \\
&\geq  r(s,k) - r(s,k+1), \label{eq:mono-k-kl-2}
\end{align}

\noindent
where inequality \eqref{eq:mono-k-kl-2} follows Lemma 4.7.2 in \cite{puterman2014markov} because \(V(s,k)\) is non-increasing in \(k\in \mathcal{K}\) and \(\hat{p}(\cdot|\cdot,k+1) \succeq \hat{p}(\cdot|\cdot,k)\) in Assumption \ref{asm:transition}. And the left hand side (LHS) of Equation \eqref{eg:mono-k-kl-1}  be rewritten as 
\begin{equation}
\textnormal{LHS} = -c_\text{r} +\beta V(0,k+1) + c_\text{r} -\beta V(0,k+2) \leq \beta V(0,k+1)-\beta c_\text{s} \leq   \frac{\beta r(0,0)}{1-\beta} - \beta c_\text{s}, \label{eq:mono-k-kl-3}
\end{equation}

\noindent
where the first inequality holds because \(V(0,k+2) \geq w(s,k+2;2)=c_\text{s}\), and the second inequality holds because \(V(0,k+1)\leq \sum_{t=0}^{\infty}\beta^t r(0,0) = r(0,0)/ (1-\beta)\). By \eqref{eq:mono-k-kl-2} and \eqref{eq:mono-k-kl-3}, we have \(\beta r(0,0)/(1-\beta) - \beta c_\text{s} \geq r(s,k)-r(s,k+1)\), which violates condition in Theorem \ref{thm:Monotone-K-KL}(a) and implies that \(a^\ast(s,k+1) = 1\) if \(a^\ast(s,k) = 1\).

It is straightforward that \(\zeta_{\text{scrap}}(k)\) is non-increasing in \(k\in\mathcal{K}\) because \(a^\ast(s,k+1)=2\) if \(a^\ast(s,k)=2\) as shown in the proof of Theorem \ref{thm:Monotone-S-KL}.

\subsection*{A.5 Proof of Proposition \ref{prop:Monotone-ValueFunction-S-IM}}
Before proving our main results, we first present a lemma  that identifies the worst distribution of the following problem:
\begin{align}
\min_{\pmb p(\cdot|s)} & \ \sum_{s'\in\mathcal{S}}p(s'|s)\nu(s') \label{A1_eq01} \\
\ \text{s.t.} & \ \sum_{s'\in\mathcal{S}}p(s'|s)=1 \nonumber \\
& \ \underline{p}(s'|s) \leq p(s'|s) \leq \bar{p}(s'|s), \ \forall s' \in\mathcal{S}  \nonumber
\end{align}
with given \(s\in\mathcal{S}\), where \(\underline{\pmb p}(\cdot|s)\) and \(\bar{\pmb p}(\cdot|s)\) are effective lower and upper bounds defined by equations \eqref{upperbound} and \eqref{lowerbound}.

\begin{lemma}
	\label{lemma:lemmaA01}
	If \(\nu(s)\) is non-increasing in \(s\in\mathcal{S}\), then  the optimal solution of the problem (\ref{A1_eq01}) (i.e., the worst distribution) \(\pmb p^\ast(\cdot|s)\) is as follows:
	\begin{equation}
	\label{lemmaA01_condition01}
	p^\ast(s'|s) =
	\left\{
	\begin{array}{ll}
	\bar{p}(s'|s) , & s' > \delta_s\\
	\underline{p}(s'|s), & s' < \delta_s, \\
	1-\sum_{s''=0}^{s^\ast-1}\underline{p}(s''|s) -\sum_{s''=s^\ast+1}^{S}\bar{p}(s''|s)  , & s' = \delta_s,\\
	\end{array}
	\right.
	\end{equation} 
	where \(\delta_s=\min\big\{\delta\in\mathcal{S}:\sum_{s'=0}^{\delta}\underline{p}(s'|s)+\sum_{s'=\delta+1}^{S}\bar{p}(s'|s) \leq 1\big\}\).
\end{lemma}
\begin{proof}
	We prove Lemma \ref{lemma:lemmaA01} by introducing a contradiction. Suppose \(\pmb p^\prime(\cdot|s)\) is an optimal solution and there exists an \(i > \delta_s\) such that \(p^\prime(i|s) < p^*(i|s)=\bar{p}(i|s)\). This implies that there exists an \(j \leq \delta_s\) such that \(p^\prime(j|s) > p^\ast(j|s)\). Let \(\Delta p = \min\big\{ p^\ast(i|s)-p^\prime(i|s), p^\prime(j|s)-p^\ast(j|s)\big\}\). We construct a new distribution  \(p^{\prime\prime}(s'|s)\) such that \(p^{\prime\prime}(s'|s)=p^\prime(s'|s)\) for \(s'\in\mathcal{S}\backslash\{ i,j\}\), \(p^{\prime\prime}(i|s)=p^\prime(i|s)+\Delta p\), and \(p^{\prime\prime}(j|s)=p^\prime(j|s)-\Delta p\). It is easy to verify that \(\pmb p^{\prime\prime}(\cdot|s)\) is feasible to problem (\ref{A1_eq01}). Because \(\nu(i)\le \nu(j)\), we have
	\begin{align*}
	\sum_{s'=0}^{S}p^{\prime\prime}(s'|s)\nu(s') &=\sum_{s' \in \mathcal{S}\backslash\{i,j\}}p^{\prime\prime}(s'|s)\nu(s') + p^{\prime\prime}(i|s)\nu(i) +p^{\prime\prime}(j|s)\nu(j)\\
	& = \sum_{s' \in \mathcal{S}\backslash\{i,j\}}p^{\prime}(s'|s)\nu(s') +\big(p^{\prime}(i|s)+\Delta p\big)\nu(i) +\big(p^{\prime}(j|s)-\Delta p\big)\nu(j)\\
	& =\sum_{s'=0}^{S}p^{\prime}(s'|s)\nu(s') + \Delta p\big(\nu(i)-\nu(j)\big)\leq \sum_{s'=0}^{S}p^{\prime}(s'|s)\nu(s').
	\end{align*}
	This means there exists a feasible solution \(\pmb p^{\prime\prime}(\cdot|s)\) that is no worse than \(\pmb p^{\prime}(\cdot|s)\). Thus there exists a contradiction.
\end{proof}

We now prove the monotonicity of the value function. We first show that \(\pmb p^\ast(\cdot|\cdot,k)\) is IFR for any given \(k \in \mathcal{K}\) if \(\underline{p}(s'|s,k)\) and \(\bar{p}(s'|s,k)\) satisfy conditions (\ref{lowerboundConstraints-S}) and (\ref{upperboundConstraints-S}) for all \((s,k)\in\mathcal{S}\times\mathcal{K}\).  Let \(s^\prime \geq s\) and  \(\delta_s\) be the same \(\delta_s\) defined in Lemma \ref{lemma:lemmaA01}.  If \(  i \leq \delta_s\), we have
\begin{equation*}
\sum_{s''=i}^{S} p^\ast(s''|s',k) = 1-\sum_{s''=0}^{i-1} \underline{p}(s''|s',k)\geq 1-\sum_{s''=0}^{i-1} \underline{p}(s''|s,k) \geq 1-\sum_{s''=0}^{i-1} p^\ast(s''|s,k) = \sum_{s''=i}^{S} p^\ast(s''|s,k),
\end{equation*}
where the first inequality is a result of condition (\ref{lowerboundConstraints-S}). If \( i > \delta_s \), we  have
\begin{equation*}
\sum_{s''=i}^{S} p^\ast(s''|s',k) = \sum_{s''=i}^{S} \bar{p}(s''|s',k) \geq \sum_{s''=i}^{S} \bar{p}(s''|s,k) \geq \sum_{s''=i}^{S} p^\ast(s''|s,k),
\end{equation*} 
where the first inequality is a result of condition (\ref{upperboundConstraints-S}). Therefore, the result follows. We can similarly show that   \(\pmb p^\ast(\cdot|\cdot,k+1) \succeq \pmb p^\ast(\cdot|\cdot,k)\) for all \( k \in\mathcal{K}\) if conditions (\ref{lowerboundConstraints-K}) and (\ref{upperboundConstraints-K}) for all \((s,k)\in\mathcal{S}\times\mathcal{K}\) are satisfied.

Having established the structure properties of transition matrices, we next prove part (a) regarding the monotonicity of \(V(s,k)\) with respect to \(s\) for all \(k \in \mathcal{K}\).

The proof is based on the robust value iteration algorithm in \cite{iyengar2005rmdp}. Let \(V^n(s,k)\) be the value function of the state \((s,k)\) at the end of iteration \(n\). We show that \(V^n(s,k)\) is non-increasing in \(s\) for all \(k \in \mathcal{K}\) in every iteration \(n\) and therefore \(V(s,k)\) is non-increasing in \(s\) for all \(k \in \mathcal{K}\) as the algorithm converges.

We prove this by induction. Let the initial values in the robust value iteration algorithm be $V^0(s,k)=0$ for all  $s\in\mathcal{S}$ and \(k\in\mathcal{K}\), then the induction hypothesis holds at the initial iteration. Assume $V^n(s,k)$ is non-increasing in $s \in \mathcal{S}$ for $k\in\mathcal{K}$ for $n = 1,\ldots,m-1$.  Let $s', s \in \mathcal{S}$ with $s' > s$. We first show that \(\pmb p^\ast(\cdot|\cdot,k)\) is IFR for any given \(k \in \mathcal{K}\). At iteration \(m\), if \(a=0\),
\begin{align}
w^m(s,k;0)  &= r(s,k) + \beta \sum_{s''\in\mathcal{S}}p^\ast(s''|s,k)V^{m-1}(s'',k)  \geq r(s^\prime,k) +  \beta \sum_{s''\in\mathcal{S}}p^\ast(s''|s,k)V^{m-1}(s'',k)  \nonumber \\
&  \geq r(s^\prime,k) + \beta \sum_{s''\in\mathcal{S}}p^\ast(s''|s',k) V^{m-1}(s'',k)  \label{eq:monotonicity-KL-s}\\
&= w^m(s^\prime,k;0),\nonumber
\end{align}
where inequality  \eqref{eq:monotonicity-KL-s} follows Lemma 4.7.2 in \cite{puterman2014markov} because \(V^{m-1}(s,k)\) is non-increasing in \(s\) by the induction hypothesis and \(\pmb p^\ast(\cdot|\cdot,k)\) is IFR given \(k\).

If \(a=1\),  \(w^m(s^\prime,k;1)=-c_\text{r}+\beta V^{m-1}(1,k+1)= w^m(s,k;1)\). Therefore, \(w^m(s,k;1)\) is non-increasing in \(s\) for all \(k\in \mathcal{K}\). We can similarly prove that \(w^m(s,k;2)\) is also non-increasing in \(s\) for all \(k\in \mathcal{K}\).

Because \(V^m(s,k)=\max_{a\in\mathcal{A}} w^m(s,k;a)\) \(\geq \max_{a\in\mathcal{A}} w^m(s^\prime,k;a) = V^m(s^\prime,k)\). Therefore, the induction hypothesis holds at iteration \(m\).  

The proof of part (b) is similar to that of part (a), and is omitted.

\subsection*{A.6 Proof of Corollary \ref{cor:worstcaseditr-IM}}
Corollary \ref{cor:worstcaseditr-IM} is a direct result of Lemma \ref{lemma:lemmaA01} and the proof is omitted.

\subsection*{A.7 Proof of Corollary \ref{cor:worstcaseditr-structure-IM}}
Corollary \ref{cor:worstcaseditr-structure-IM} has already been proved in Theorem \ref{thm:Monotone-S-IM} and the proof is omitted.

\subsection*{A.8 Proof of Theorem \ref{thm:Monotone-S-IM}}
The proof is similar to the proof of Theorem \ref{thm:Monotone-S-KL}.

\subsection*{A.9 Proof of Theorem \ref{thm:Monotone-K-IM}}
The proof is similar to the proof of Theorem \ref{thm:Monotone-K-KL}.

\subsection*{A.10 Proof of Theorem \ref{thm:sensitivity}}
Suppose we solve the two problems simultaneously using the robust value iteration algorithm. We first show that starting with a value of 0 for all states in both problems, at the end of each iteration of the algorithm, the value function of \(\Lambda_1\) will be greater than or equal to the value function of \(\Lambda_2\) for each state. Let \(V_i^n(s,k)\) be the value function of the state \((s,k) \in \mathcal{S}\times \mathcal{K}\) of problem \(\Lambda_i\) at the end of iteration \(n\). Let  \(\mathcal{U}_{sk}^i\),  \(\pmb{p}_i(\cdot|s,k)\), and \(\pmb{p}_i^\ast(\cdot|s,k)\) denote the ambiguity set, transition probability, and the worst transition probability for state \((s,k) \in \mathcal{S}\times \mathcal{K}\) of problem \(\Lambda_i\), respectively. 

We prove this by induction. Since \(V_1^0(s,k)=V_2^0=0\) for \((s,k) \in \mathcal{S}\times \mathcal{K}\), the induction holds at the initial iteration. Now, assume that \(V_1^n(s,k)\ge V_2^n(s,k), (s,k) \in \mathcal{S}\times \mathcal{K}\), for \(n=1,\ldots,m-1\). Then we want to show that \(V_1^m(s,k)\ge V_2^m(s,k), (s,k) \in \mathcal{S}\times \mathcal{K}\).  At iteration \(m\), if \(a=0\), we have 

	\begin{align}
	w_1^m(s,k;0) &= \min_{\pmb p_1(\cdot|s,k)\in\mathcal{U}^1_{sk}}r(s,k)+\beta\sum_{s'\in\mathcal{S}}p_1(s'|s,k)V_1^{m-1}(s',k) \nonumber\\
	&\ge \min_{\pmb p_1(\cdot|s,k)\in\mathcal{U}^1_{sk}}r(s,k)+\beta\sum_{s'\in\mathcal{S}}p_1(s'|s,k)V_2^{m-1}(s',k) \label{eq:sensitivity1}\\
	&\ge \min_{\pmb p_2(\cdot|s,k)\in\mathcal{U}^2_{sk}}r(s,k)+\beta\sum_{s'\in\mathcal{S}}p_2(s'|s,k)V_2^{m-1}(s',k) \label{eq:sensitivity2}\\
	&=w_2^m(s,k;0), 
	\end{align}

where  inequality \eqref{eq:sensitivity1} follows the induction hypothesis and the  inequality in \eqref{eq:sensitivity2} follows $\mathcal{U}_{sk}^1\subseteq \mathcal{U}_{sk}^2$. 

If \(a=1\), \(w_1^m(s,k;1) = -c_\text{r}+V_1^{m-1}(s,k) \ge -c_\text{r}+V_2^{m-1}(s,k) = w_2^m(s,k;1)\), since \(V_1^{m-1}(s,k)\ge V_2^{m-1}(s,k)\) for \((s,k) \in \mathcal{S}\times \mathcal{K}\) by the induction assumption.  If \(a=2\), \(w_1^m(s,k;2) =  w_2^m(s,k;2) = c_\text{s}\). Since \(V_1^m(s,k)=\max_{a\in\mathcal{A}} w_1^m(s,k;a) \geq \max_{a\in\mathcal{A}} w_2^m(s,k;a) = V_2^m(s,k)\) for all \((s,k) \in \mathcal{S}\times \mathcal{K}\),
the induction hypothesis holds at iteration \(m\). Because  the value function of \(\Lambda_1\) is always greater than or equal to that of \(\Lambda_2\) at each iteration of the value-iteration algorithm, the optimal value function of \(\Lambda_1\) is greater than or equal to that of \(\Lambda_2\).

Next, we prove part (a) that \(\zeta_{\text{rm}}^1(k)\le \zeta_{\text{rm}}^2(k)\)  for  \(k < k_1^\ast\), where $\tilde{s} = \zeta_{\text{rm}}^1(k)-1$ and \(k_1^\ast\) is the threshold defined in Theorems \ref{thm:Monotone-K-KL} and \ref{thm:Monotone-K-IM} for \(\Lambda_1\).  This is equivalent to show \(a_2^{\ast}(\tilde{s},k) = 0\) if \(a_1^{\ast}(\tilde{s},k) = 0\) and \(\hat{p}(\tilde{s}|\tilde{s},k) = 0\) for all \(k < k_1^{\ast}\). We prove this by introducing a contradiction. Suppose  \(a_2^{\ast}(\tilde{s},k) = 1\) when \(a_1^{\ast}(\tilde{s},k) = 0\). Then we have \(w_1(\tilde{s},k;0) \geq w_1(\tilde{s},k;1)\) and \(w_2(\tilde{s},k;1) \ge w_2(\tilde{s},k;0)\). Thus, we have 
\begin{equation}
\label{eq:contradition-1}
w_1(\tilde{s},k;0) - w_2(\tilde{s},k;0) \geq w_1(\tilde{s},k;1)  - w_2(\tilde{s},k;1).
\end{equation} 
The left-hand-side (LHS) of (\ref{eq:contradition-1}) can be rewritten as
\begin{align}
\textnormal{LHS} &= r(\tilde{s},k)+\beta\sum_{s'\in\mathcal{S}}p_1^*(s'|\tilde{s},k)V_1(s^\prime,k) - r(\tilde{s},k)-\beta\sum_{s'\in\mathcal{S}}p_2^{\ast}(s'|\tilde{s},k)V_2(s',k)\nonumber\\
&= \beta p_1^*(\tilde{s}|\tilde{s},k)V_1(\tilde{s},k) + \beta \sum_{s'\geq \tilde{s}+1}p_1^*(s'|\tilde{s},k)V_1(s',k) -\beta\sum_{s'\in\mathcal{S}}p_2^{\ast}(s'|\tilde{s},k)V_2(\tilde{s},k)\label{eq:contradition-2}\\
&= \beta p_1^*(\tilde{s}|\tilde{s},k)V_1(\tilde{s},k) + \beta  \big(1-p_1^*(\tilde{s}|\tilde{s},k)\big)w_1(\tilde{s},k;1) -\beta w_2(\tilde{s},k;1)\label{eq:contradition-3}\\
&= \beta\big(w_1(\tilde{s},k;1)- w_2(\tilde{s},k;1)\big) \label{eq:contradition-4}, 
\end{align}

\noindent
The equality \eqref{eq:contradition-2} is obtained by simply rearranging terms. The equality \eqref{eq:contradition-3} holds because \(a_2^{\ast}(\tilde{s},k) = 1\) and  \(a_2^{\ast}(s',k) = 1\) for all \(s'\geq \tilde{s}\) following Theorem \ref{thm:Monotone-S-KL}. The equality \eqref{eq:contradition-4} holds because \(p^\ast_1(\tilde{s}|\tilde{s},k)=0\).  By \eqref{eq:contradition-1} and \eqref{eq:contradition-4}, we must have \(\beta\big(w_1(\tilde{s},k;1)- w_2(\tilde{s},k;1)\big) \geq w_1(\tilde{s},k,1)  - w_2(\tilde{s},k,1)\). However, the inequality does not hold since \(\beta < 1\) and leads to a contradiction. Therefore, we have \(a_2^{\ast}(\tilde{s},k) = 0\) if \(a_1^{\ast}(\tilde{s},k) = 0\) and $p_1^*(\tilde{s}|\tilde{s},k) = 0$.

Next, we prove part (b) that \(\zeta_{\text{scrap}}^1(k)\ge \zeta_{\text{scrap}}^2(k)\) for \(k\ge k_1^\ast\). This is equivalent to show \(a_2^{\ast}(s,k)=2\) if \(a_1^{\ast}(s,k)=2\). Because \(V_1(s,k)=w_1(s,k;2) = c_\text{s} = w_2(s,k;2) \le V_2(s,k)\) and \(V_1(s,k)\ge V_2(s,k)\) by the induction, we have  \( V_2(s,k)=w_2(s,k;2)\) and thus  \(a_2^{\ast}(s,k)=2\).

Since \(a_2^{\ast}(s,k)=2\) if \(a_1^{\ast}(s,k)=2\) for  \(k\ge k_1^\ast\), part (c) follows.

\subsection*{B.1 Experiment Parameters}

The following table provides the experiment parameters used in the experiment that examines the existence of control limit policies when the condition of Theorem \ref{thm:Monotone-K-KL}(a) is violated. Note that for easy parameter control, we redefine the reward as \(r(s,k)=a_0 - a_1k-a_2s\). Parameter values are drawn from their respective uniform distributions.

\begin{table}[htbp]
	\centering
	\begin{tabular}{ccccccc}
		\hline
		$a_0$ & $a_1$ & $a_2$  & \(c_r\) & \(c_s\) & \(\theta\) & \(\beta\)      \\ \hline
		$U(10, 50)$ &  $U(1,15)$ &  $U(1,15)$  &  \(U(0,10)\)    & \(U(0,10)\)    & \(U(0,2)\) & \(U(0.01,0.99)\) \\ \hline
	\end{tabular}
\end{table}

\end{document}